\documentclass[12pt, reqno]{amsart}
\usepackage{amssymb,amsmath,amsthm,newlfont,color}
  \usepackage{psfrag}
\usepackage[T1]{fontenc}
\usepackage{a4wide}
\usepackage[dvipsnames]{xcolor}
\usepackage{amssymb}
\usepackage{amsmath}
\usepackage{amsthm}
\usepackage[all]{xy}
\usepackage{accents}
\usepackage{comment}
\usepackage{eurosym}
\usepackage{graphicx}  
\usepackage[bottom]{footmisc}
\usepackage[colorlinks=true,linkcolor=blue,urlcolor=blue,citecolor=blue]{hyperref}
\usepackage{nomencl}
\usepackage{glossaries}
\usepackage{tikz}
\usepackage[abs]{overpic}
\usepackage{pict2e}
\usepackage{vmargin}
\usepackage{enumitem}
\usepackage{multicol}
\newcommand\C{\mathbb{C}}

\newcommand\R{\mathbb{R}}

\newcommand\U{\mathbb{D}}

\newcommand\V{\mathcal{V}}

\renewcommand\d{\partial}
\newcommand\dbar{\overline{\partial}}

\newcommand\Bergman{\mathcal{A}_{\alpha}^2}
\newcommand\Bergmaninf{\mathcal{A}_{\alpha}^\infty}

\newcommand\N{\mathbb{N}}
\DeclareMathOperator{\Hol}{Hol}

\newcommand{\dst}{\displaystyle}
\newcommand{\dist}{\operatorname{dist}}

\newtheorem{theorem}{Theorem}[section]
\newtheorem*{theorem*}{Theorem} 
\newtheorem{lemma}[theorem]{Lemma}

\newtheorem{definition}[theorem]{Definition}

\begin{document}
\title{Multiple Sampling and Interpolation in Bergman Spaces}

\author[D. Aadi]{D. Aadi}
\address{D. Aadi, Mohammed V University in Rabat, Faculty of sciences, CeReMAR -LAMA- B.P. 1014 Rabat, Morocco}
\email{driss\_aadi@um5.ac.ma}

\author[C. Cruz]{C.  Cruz}
\address{Universitat de Barcelona, Departament de Matem\`atiques i Inform\`atica, Gran Via 585, 08007-Barcelona, Spain}
\email{ccruz@ub.edu}

\author[A. Hartmann]{A. Hartmann}
\address{A. Hartmann, Univ. Bordeaux, CNRS, Bordeaux INP, IMB, UMR  
5251,  F-33400, Talence, France}
\email{Andreas.Hartmann@math.u-bordeaux.fr}

\author[K. Kellay]{K. Kellay}
\address{K. Kellay, Univ. Bordeaux, CNRS, Bordeaux INP, IMB, UMR 5251,  
  F-33400, Talence, France}
\email{kkellay@math.u-bordeaux.fr}


\thanks{The second  author is supported by the APIF project. The research of the third and fourth authors is partly supported by the ANR-18-CE40-0035 project 
}

\maketitle

\begin{abstract}
We study multiple sampling and interpolation problems with unbounded multiplicities in the weighted Bergman space, 
both in the hilbertian case $p=2$ and the uniform case
$p=+\infty$. 
\end{abstract}
\section{Introduction}\label{section1}

Interpolating and sampling sequences have been studied in a broad variety of settings. 
We refer to the books \cite{Seip,HKZ00,DS04} for an account on these problems.
The particular situation of the Bergman space had been completely solved by Seip in \cite{Sei93} 
using density conditions, and in more general Hilbert spaces of analytic functions by Berndtsson and Ortega-Cerd\`a in \cite{BO95}. Subsequently, the case of multiple interpolation { {(but not
sampling)}} with uniformly
bounded multiplicites had been studied for instance
by Krosky and Schuster \cite{KS} using { also} extremal functions (we mention related work by the third named author who considered in \cite{Hart01} finite unions of Bergman interpolating sequences based on extremal functions, 
the case of multiple interpolation being in a sense a limite case of finite unions).
In view of his density characterizations, Seip's results imply that there are no simultaneous 
sampling and interpolating sequences in the Bergman spaces.

In the Fock space, besides considering the case of simple interpolation and sampling problems, 
Seip -- in particular with Brekke -- was interested in the situation of higher multiplicities.
Again, the density conditions obtained by these
authors imply that there are no simultaneous sampling and interpolating sequences, neither in
the simple case nor in the multiple case. Brekke and Seip in \cite{BS}  also asked whether there could be
simultaneous sampling and interpolating sequences when the uniform boundedness condition on
the multiplicities is relaxed. In \cite{BHKM} it was shown that at least when the multiplicities
tend to infinity, this is not possible (see also \cite{AAR} for the case of bounded multiplicites in the weighted Fock space).

One difficulty occuring in the case of unbounded multiplicities is the lack of a reasonable
definition of densities. In \cite{BHKM}, the authors introduce covering and separation conditions
related with critical radii suitably related with the multiplicities to circumvent densities. 
Though those conditions do not
characterize multiple interpolation and sampling, they get in a sense closer and closer to a 
characterization when the multiplicities grow (indeed the difference between necessary and 
sufficient conditions of the radii remains bounded while the radii tend to infinity).

Bergman and Fock spaces share many properties, and techniques often translate
from one setting to the other. The aim of this paper is to study the situation 
concerning multiple interpolating and sampling with unbounded
multiplicities in the Bergman space. New difficulties and challenges appear in order to
adapt the situation from the underlying euclidean metric in the Fock space to 
pseudohyperbolic metric in the Bergman space. While this might be rather direct for
simple interpolation and sampling the situation requires a quite delicate analysis of
the criticial radii in the pseudohyperbolic metric in particular when the multiplicities
are not uniformly bounded. On the technical side, replacing the incomplete $\Gamma$-function
by the incomplete $\beta$-function gives rise to other difficulties.

It is mentionable that generalized interpolation problems (but not sampling problems)
have been considered long ago in the Hardy space for which a complete answer is given
by the so-called generalized Carleson condition (see \cite{Nik, V}). In this situation, the
case of interpolating sequences with unbounded multiplicities is completely understood
(see also earlier work by Vinogradov-Rukshin \cite{VR}). 

Without claiming exaustivity, we finish this first tour on multiple interpolation problems, mentioning work
on interpolating sequences with uniformly bounded multiplicity in the Korenblum space, see \cite{Mas99} and for weighted spaces of entire functions see \cite{Oun07a,Oun07b}. \\

We now introduce the necessary notation. Let  $\alpha>-1$, we consider the $L^2$ weighted Bergman space 
\begin{equation*}
    \mathcal{A}^2_\alpha=\left\{f\in \Hol
(\U):\|f\|^2_{\alpha,2}=\int_\U|f(z)|^2dA_\alpha(z)<+\infty\right\},
\end{equation*}
where $dA_\alpha(z)=(1+\alpha)(1-|z|^2)^{\alpha}dxdy/\pi$, $z=x+iy$. The space $\mathcal{A}_\alpha^2$ is a reproducing kernel Hilbert space with the scalar product
\begin{equation*}
    \langle f,g\rangle:=\int_\U f(z)\overline{g(z)}dA_\alpha(z).
\end{equation*}

The standard monomial orthonormal basis for $\Bergman$ is given by
\begin{equation}\label{lemmaorthonormal}
    e_n(z)=\sqrt{\frac{\Gamma(n+2+\alpha)}{n!\Gamma(2+\alpha)}}z^n,\qquad n\geq0,
\end{equation}
where $\Gamma(s)$ stands for the usual Gamma function 
 \cite[p.4]{HKZ00}. Thus, see for instance \cite[p.5]{HKZ00} the  reproducing kernel  $\Bergman$ is  
\begin{equation*}
    K_w(z)=\sum_{j\geq 0}\overline{e_j(w)}e_j(z)=\frac{1}{(1-\overline{w}z)^{\alpha+2}},
\end{equation*}
and the normalized Bergman kernel is    $k_w(z)=K_w(z)/\|K_w\|_{\alpha,2}$. 
The reproducing kernel gives rise in a standard way to a growth condition in the Bergman space:
\begin{equation}\label{Croissance}
 |f(\lambda)|^2=|\langle f, K_{\lambda}\rangle|^2\le \left(\frac{1}{1-|\lambda|^2}
 \right)^{\alpha+2} \|f\|_{\alpha,2}^2, \quad f\in \Bergman, \lambda\in \U.
\end{equation}

We consider the M\"obius transform 
\begin{equation}\label{Mobius}
 \varphi_\lambda(z)=\frac{\lambda-z}{1-\overline{\lambda}z}
\end{equation}
and the {{isometric}} translation operators in $\mathcal{A}^2_{\alpha}$ given by 
$T_\lambda f(z)=\left[\varphi_\lambda'(z)\right]^\frac{2+\alpha}{2}f(\varphi_\lambda(z))$, 
i.e.,
\begin{align*}
    T_\lambda : \  \mathcal{A}^2_\alpha&\longrightarrow\mathcal{A}^2_\alpha\\
    f&\longrightarrow T_\lambda f:=\left[\frac{|\lambda|^2-1}{(1-\overline{\lambda} \cdot)^2}\right]^\frac{2+\alpha}{2}f(\varphi_\lambda(\cdot)).
\end{align*}
Notice that $\varphi_\lambda$ and $T_\lambda$ are involutions, in fact, $T_\lambda$ is a self-adjoint operator.\\

In the Bergman space, the underlying metric on the unit disk is the pseudohyperbolic
distance which is defined {\it via} the already mentioned M\"obius transform \eqref{Mobius}:
\[
 \rho(u,v)=|\varphi_u(v)|,\quad u,v\in \U.
\]
We also associate the pseudohyperbolic disk with this distance:
$D(\lambda,r)=\{z\in\U:\rho(\lambda,z)<r\}$ for $\lambda\in \U$ and $0<r<1$.
In order to be interpolating in the Bergman space, a sequence has to be separated
in this metric (see \cite{Sei93}). Also, in order to be sampling (at least in the Hilbertian case
under consideration here), it can be deduced that
any pseudohyperbolic neighborhood with fixed radius can contain at most a uniformly
bounded number of points (Carleson measure condition).
\\

In order to better understand multiple interpolation and sampling problems
we shall comment a little bit more on the multiplicity one situation. 
In this case, 
given a set of points $\Lambda\subset\U$,
one is interested in the values
of a function in given points $f(\lambda)$, $\lambda\in \Lambda$. When $\Lambda$
is separated in the pseudohyperbolic metric, it can be shown that for
$f\in \mathcal{A}^2_{\alpha}$, we have
\begin{equation}\label{SimpleInt}
 \sum_{\lambda\in\Lambda}\frac{|f(\lambda)|^2}{\|K_\lambda\|^2} 
= \sum_{\lambda\in\Lambda}(1-|\lambda|^2)^{2+\alpha}|f(\lambda)|^2
<\infty.
\end{equation}
Conversely, if the sequence $\Lambda$ is sufficiently separated, if can be shown that every
sequence of values $(v_{\lambda})_{\lambda\in\Lambda}$ the square of which is summable
against the weight $(1-|\lambda|^2)^{2+\alpha}$ can be interpolated by a function
$f\in \mathcal{A}^2_{\alpha}$.
Observe that 
\[
 \langle f,T_{\lambda}e_0\rangle=
 T_{\lambda}f(0)=\left[\frac{|\lambda|^2-1}{(1-\overline{\lambda}\times 0)^2}\right]^\frac{2+\alpha}{2}f(\varphi_\lambda(0))=(|\lambda|^2-1)^\frac{2+\alpha}{2}f(\lambda),
\]
so that \eqref{SimpleInt} translates to
\[
 \sum_{\lambda\in\Lambda}|\langle f,T_{\lambda}e_0\rangle|^2<+\infty.
\]

Notice that, intuitively, a sequence of interpolation must be sufficiently sparse, and should be a set of zeros of a holomorphic function (at least up to one point). Analogously, any sampling sequence should have sufficiently big density and is in general not a set of zeros (except in spaces which admit
complete interpolating sequences). This naive relation points at the connection between our problems and uniqueness questions which will be studied in the next Section \ref{setionzeros}.
\\

Let us now switch to the multiple case.
Instead of studying only the values of a function in the points of the
given sequence, it is natural to consider germs of functions in those points, i.e.\ to consider
also derivatives of the function up to a certain order depending on the point (Hermite type
interpolation).
As long as the multiplicites are uniformly bounded, the definition of the target space
can be based on point evaluations and their derivatives with suitable weights. However,
when we allow multiplicities to grow to infinity, the weights and constants have to 
be chosen in a very precise way. 
In the Hilbertian case, since $(e_j)_{j\ge 0}$ is an orthonormal basis and $T_{\lambda}$
an isometry, it is natural to consider $\langle f,T_{\lambda}e_j\rangle$, where $\lambda\in
\Lambda$ and $j$ is bounded by the multiplicity.
With this in mind, we can now define sampling and interpolation in the general case. 	
Given a set of points $\Lambda\subset\U$ with multiplicity $m_\lambda$, we call \textit{divisor} a set of pairs $X=\{(\lambda,m_\lambda) \}_{\lambda\in\Lambda}$. 
\begin{definition}
A divisor $X$ is sampling for $\Bergman$ if there exists a constant $C>0$, such that for all $f\in\Bergman$
\begin{equation*}
\frac{1}{C}\sum_{\lambda\in\Lambda}\sum_{j<m_\lambda}\left|\langle f,T_\lambda e_j\rangle\right|^2\leq\|f\|_{\alpha,2}^2\leq C\sum_{\lambda\in\Lambda}\sum_{j<m_\lambda}\left|\langle f,T_\lambda e_j\rangle\right|^2.
\end{equation*}
\end{definition}

Note that $P_{\lambda}f=\sum_{j<m_{\lambda}}\langle f,T_{\lambda}e_j\rangle
T_{\lambda}e_j$ is an orthogonal projection and it can be shown that
\begin{equation}\label{p5*}
\ker P_{\lambda}=N_{\lambda,m_\lambda}^{2,\alpha}:=\{f\in\Bergman: f^{(j)}(\lambda)=0, \quad\forall j<m_\lambda\},
\end{equation}
(see also equation (4) in \cite[p.114]{BS} where this matter is discussed in the Fock space),
so that in particular
\[
 \sum_{j<m_\lambda}\left|\langle f,T_\lambda e_j\rangle\right|^2
 =\|f\|^2_{\mathcal{A}^2_{\alpha}/N_{\lambda,m_{\lambda}}^{2,\alpha}}.
\] 

It is usefull to recall the weaker notion of uniqueness. A divisor $X$ is called a {\it uniqueness divisor}
(or simply $X$ is uniqueness) if every function vanishing up to the order $m_{\lambda}-1$ is
$\lambda$ is necessarily the zero function. Clearly, a sampling divisor is uniqueness, but the
converse is in general not true.
\\

The above definition gives rise to a natural definition of the following target space
needed for interpolation.
\begin{equation*}
\ell^2(X)=\left\{(v_\lambda^j)_{\lambda\in\Lambda,\ 0\leq  j<m_\lambda}: \|v\|^2_2:= \sum_{\lambda\in\Lambda}\sum_{j<m_\lambda}|v_\lambda^j|^2<\infty\right\}.
\end{equation*}
\begin{definition}
The divisor $X$ is interpolating for $\Bergman$ if for all sequences $v\in \ell^2(X)$, there exists $f\in\Bergman$ such that
\begin{equation*}
\langle f, T_\lambda e_j\rangle=v_\lambda^j.\qquad (\lambda\in \Lambda,j<m_\lambda)
\end{equation*}
\end{definition}

Note again that the above interpolation condition is equivalent to interpolation by germs of 
$f$ in $\lambda$ up to the order $m_{\lambda}-1$ (see equation (4) in 
\cite[p.114]{BS}). Clearly, if $f$ interpolates $v$ in the above way, then
$\sum_{j<m_\lambda}|v_\lambda^j|^2=\|f\|^2_{\mathcal{A}^2_{\alpha}/N_{\lambda,m_{\lambda}}^{2,\alpha}}$. The reinterpretation in terms of quotient norms
will be useful later when considering the situation in $\mathcal{A}^{\infty}_{\alpha}$.\\

In the case of the classical Fock space, multiple interpolation and sampling was related to
some critical radius. More precisely, since in the Fock space the underlying metric is
euclidean, given a multiplicity $m_{\lambda}$ in a point $\lambda\in\C$ the ``influence
zone'' of $\lambda$ meaning
the knowledge of $f(\lambda)$,...,$f^{m_{\lambda}-1}(\lambda)$ --- or
equivalently that of $\langle f,T_{\lambda}e_0\rangle$,...,$\langle f,T_{\lambda}e_{m_{\lambda}}
\rangle$ --- was the euclidean disk $D_e(\lambda,\sqrt{m_{\lambda}})=
\{z\in\C:|z-\lambda|<\sqrt{m_{\lambda}}\}$ (see \cite{BHKM}). This corresponds more
or less to redistributing the multiplicity in a regular way in an euclidean disk.
It is a priori not so clear how to define this redistribution in the pseudohyperbolic case in particular
when we authorize the multiplicity to tend to infinity. The corresponding critical radius 
appears in the following overlap condition; it will be clear from later discussions where this radius 
comes from. Our sampling and interpolating conditions are all expressed with respect to this
critical value (slightly increasing or decreasing it). We will need an overlap condition that 
we introduce now.

\begin{definition}\label{FiniteOvl}
A divisor $X$ satisfies the finite overlap condition for $\Bergman$ if 
\begin{equation*}
    S_X=\sup_{z\in\U}\sum_{\lambda\in\Lambda}\chi_{ D\left(\lambda,\sqrt{\frac{m_\lambda}{m_\lambda+\alpha+1}}\right)}(z)<\infty.
\end{equation*}
\end{definition}

We should mention that this finite overlap condition is intimately
related to the Carleson measure condition.
\\

Now we are in a position to state the geometric condition for sampling divisors. 
\begin{theorem}\label{ThmSamp2}
Let $\alpha>-1$.
\begin{enumerate}[label=(\alph*)]
\item If $X$ is a sampling divisor for $\Bergman$, then $X$ satisfies the finite overlap condition  and {%
there exists $0<C_X<\alpha+1$ } such that
\begin{equation*}
\bigcup_{\lambda\in\Lambda}  D\left(\lambda,\sqrt{\frac{m_\lambda+C_X}{m_\lambda+\alpha+1}}\right)=\U.
\end{equation*}
\item Conversely, suppose the divisor $X$ satisfies the finite overlap condition. 
There is a constant $C>1$ depending on $S_X$ 
 such that if for some compact $K$ of $\U$ we have
\begin{equation*}
\bigcup_{\lambda\in\Lambda, m_\lambda>C}  D\left(\lambda,\sqrt{\frac{m_\lambda-C}{m_\lambda+\alpha+1}}\right)=\U\setminus K,
\end{equation*}
then $X$ is a sampling divisor for $\Bergman$.
\end{enumerate}
\end{theorem}
This theorem tells us that if disks with slightly smaller radii than the critical one already
cover the unit disk, then we have a sampling divisor. And if a divisor
is sampling then at least disks with slightly bigger radii cover the unit disk (up to a compact set).

One could be tempted to complain about the constant $C$ appearing in (b) above. Note that
the theorem is completely general and applies even in the case of uniformly bounded
multiplicites where the result proved in \cite{BS} requires density conditions. So there is
no hope getting a sufficient condition only from the covering without additional 
conditions for instance on the critical radius.
\\

In the analogous situation for interpolating divisors the covering condition is replaced by 
a separation condition of disks with slightly bigger or smaller radii than the criticial ones.
\begin{theorem}\label{ThmInter2}
Let $\alpha>-1$.
\begin{enumerate}[label=(\alph*)]
\item 
If $X$ is an interpolating divisor for $\Bergman$, then there exists $C_X>0$ such that the hyperbolic disks $$\left\{ D\left(\lambda,\sqrt{\frac{m_\lambda-C_X}{m_\lambda+\alpha+1}}\right)\right\}_{\lambda\in\Lambda,m_\lambda>C_X}$$ are pairwise disjoint.
\item 
Conversely, if for some  $C_X$ such that   $(\alpha+1)(1-e^{-1})<C_X<\alpha+1$,   the hyperbolic disks  $$ \left\{ D\left(\lambda,\sqrt{\frac{m_\lambda+C_X}{m_\lambda+\alpha+1}}\right)\right\}_{\lambda\in\Lambda}$$ are pairwise disjoint, then $X$ is an  interpolating divisor for $\Bergman$.
\end{enumerate}
\end{theorem}

Notice that the separation condition appearing in the statement (a) implies the finite overlap
condition (the overlap is actually void), which is again related to the Carleson measure condition.

Again, we should point out that additional conditions on the constant $C$ are required in
(b) since the theorem is completely general covering the case when the multiplicities are
uniformly bounded in which case the result \cite{BS} involves again density conditions. So,
separation alone for $C$ arbitrary close to 0  cannot be sufficient for interpolation.
\\

{ Concerning both Theorems \ref{ThmSamp2} and \ref{ThmInter2}, we would also like to 
emphasize the fact that densities, even if they provide characterizations for simple 
or uniformly bounded
multiplicites, are hard to compute in a general situation (particularly in the 
pseudohyperbolic metric), while our overlapping and separation conditions are much easier
to apprehend.}
\\

{ {Here is another observation: in case $X$ is a sampling divisor, the finite overlap 
condition is necessary. In case $X$ is an interpolating divisor, we get a separation condition,
which obviously also implies the finite overlap (there is actually no overlap and
now $S_X=1$). So in both
cases, the area of pseudohyperbolic disks centered at $\lambda$ and with radius 
comparable to $\sqrt{(m_{\lambda}-C)/(m_{\lambda}+\alpha+1)}$ add up to a finite sum, 
yielding the following Blaschke type condition which seems new:
 \begin{eqnarray}\label{Blaschke}
 \sum_{\lambda\in\Lambda}(m_{\lambda}(1-|\lambda|^2))^2<+\infty.
\end{eqnarray}
}}

  The result which affirms that in the Fock space (with Gaussian weight) there are no Riesz bases (simultaneously interpolating and sampling)  is quite expensive to obtain. First in \cite{Sei93} for simple multiplicity, and later for uniformly bounded multiplicity  in \cite{BS} this requires the density characterizations of interpolating and sampling sequences. More recently the third and fourth authors discussed this problem in \cite{BHKM} when the multiplicities go to infinity. In this case, there are no characterizations available, but the gap between necessary conditions for multiple sampling and multiple interpolation (given by the corresponding results to Theorems \ref{ThmSamp2} and \ref{ThmInter2}), together with some geometric lemma, allowed to conclude. The situation in Bergman spaces is dramatically simpler. 
The main reason is due to the fact that the multiplier algebra for standard Bergman spaces is not trivial and contains bounded analytic functions in the unit disk. As a consequence, any interpolating divisor is a zero divisor (pick a function vanishing in all the points $\lambda$ up to the
order $m_{\lambda}-1$ except for one point $\lambda_0$ in which we interpolate the value 1, 
then multiply the interpolating function by $b_{\lambda_0}^{m_{\lambda}}$), and can
thus not be sampling since sampling divisors are uniqueness.
 \\

In view of the above discussions, the central role of zero and uniqueness sets should be
clear. In this connection, we will formulate here a necessary condition for zero divisors which 
does not seem to follow from those known so far. A fairly precise results on zero sets
in $\mathcal{A}$

\begin{theorem}\label{thmcovering2}
Let $\alpha>0$, $\varepsilon>0$, and $X=\{(\lambda,m_\lambda)\}_{\lambda\in\Lambda}$ be a divisor such that
\begin{equation}\label{eqcovering2}
    \bigcup_{\lambda\in\Lambda} D\left(\lambda,\sqrt{\frac{m_\lambda}{m_\lambda+\alpha+2+\varepsilon}}\right)=\U\setminus K
\end{equation}
for some compact set $K\subset \U$.
Then $X$ is uniqueness divisor for $\mathcal{A}^{2}_{\alpha}$.
\end{theorem}
We recall that a uniqueness divisor is a non zero divisor.\\

Without going into more details in this introduction, we mention that the same problems can be considered in the uniform norm: let $\alpha>0$, and define
\begin{equation*}
    \mathcal{A}^{\infty}_\alpha=\left\{f\in \Hol
(\mathbb{D}):\|f\|^2_{\alpha,\infty}:=\sup_{z\in\mathbb{D}} (1-|z|_{2})^{\frac{\alpha}{2}}|f(z)|<+\infty\right\}.
\end{equation*}  

In this setting, the results are completely analogous -- replacing essentially
$\alpha+1$ by $\alpha$ in the theorems cited above -- and will be discussed in Section 4.
Note that it follows immediately from \eqref{Croissance} that $\mathcal{A}^{2}_{\alpha}
\subset \mathcal{A}^{\infty}_{\alpha+2}$ which allows to connect some results between
both situations. However, in general the results in $\mathcal{A}^{\infty}_{\alpha}$ do
not follow immediately from those for $\mathcal{A}^2_{\alpha}$ and the proofs have to 
be rerun. We also point out a curious phenomenon. Indeed the above embedding works
for $\alpha>-1$, but there is no $A^2_{\beta}$ which
embeds into $A^{\infty}_{\alpha}$ when $\alpha\in (0,1]$.

{ {We would also like to comment on the uniqueness result Theorem \ref{thmcovering2} and
its corresponding result Theorem \ref{thmcovering} below in $\mathcal{A}^{\infty}_{\alpha}$
(which has essentially the same statement just
replacing $\alpha+2$ by $\alpha$). In \cite{Sei95}, Seip
gave fairly precise sufficient and necessary conditions exhibiting a small gap between
these. His conditions are based on the Korenblum density which is difficult to check in general.
The condition appearing in \eqref{eqcovering2} (or in \eqref{eqcovering} below) yields 
maybe a more transparent necessary condition for zero divisors
(for $X$ to be a zero divisor it is necessary
that the covering condition \eqref{eqcovering2} does not hold for any compact 
$K$ and any $\varepsilon>0$).}}
\\

The paper is organized as follows: in Section 2,  we prove the  $\mathcal{A}^{\infty}_{\alpha}$ uniqueness Theorem.   Section 3  is devoted to the proofs of Theorem 
\ref{ThmSamp2} and \ref{ThmInter2} respectively. In Section 4 we discuss the uniform case.\\

Throughout this paper we use the following notations :
\begin{itemize}
\item $A\lesssim B$ means that there is an absolute constant $C$ such that $A \le CB$. 
\item $A\asymp B$  if both $A\lesssim B$ and $B\lesssim A$. 
\end{itemize}

\medskip

\noindent{\bf Acknowledgements:}
The authors would like to thank Omar El Fallah, Xavier Massaneda and Joaquim Ortega-Cerd\`a 
for helpful discussions.

\section{Zeros and a Jensen type formula}\label{setionzeros}

In this section we will prove a uniqueness result for $\mathcal{A}^{\infty}_{\alpha}$ (which will
also be useful for the case $p=2$). 

Let us introduce  the invariant measure on the unit disk $d\V(z)=(1-|z|^2)^{-2}dm(z)$, where $m$ is the normalized Lebesgue measure such that $m(\U)=1$. It it well-known that $d\V$ is invariant
under M\"obius transforms. We mention that a direct calculation shows that
\[
 \int_{D(\zeta,r)}d\V(z)=\int_{D(0,r)}d\V(z)=\frac{r^2}{1-r^2}, \quad 0<r<1.
\]
We observe here that
morally speaking, the critical  radius $r$ has to be chosen more or less in such a way that
this mass corresponds to the multiplicity. To be more precise, and as we will see below (see \eqref{rlambda} below), the critical radius has to be chosen {\it via} redestributing the mass of the laplacian of the logarithm of an
$\mathcal{A}^{\infty}_{\alpha}$-function (as was done in the 
euclidean metric appearing in the setting of the Fock space).
\\

We will discuss the situation here in $\mathcal{A}^{\infty}_{\alpha}$ which requires $\alpha>0$ contrarily to $\mathcal{A}^2_{\alpha}$ where $\alpha>-1$.

Now, in the spirit of our observation above (and in particular \eqref{rlambda} below), for fixed $\varepsilon>0$, let 
\begin{equation}\label{CritValInfty}
 r_{\lambda}=  r_{\lambda,\alpha,\varepsilon}:=\sqrt{\frac{m_{\lambda}}{m_{\lambda}+\alpha+\varepsilon}}
\end{equation}
and set $D_{\lambda}=D(\lambda,r_{\lambda})$. (In a sense, the critical radius in $\mathcal{A}_{\alpha}^{\infty}$ corresponds to the situation when $\varepsilon=0$.)

\begin{lemma}\label{Lem2.1}
Let $\alpha>0$, if $X=\{(\lambda_k,m_k)\}_{k\geq1}$  is a zero divisor for $\Bergmaninf$, then
for $\varepsilon>0$,
\begin{equation*}
    \int_{D(0,r)}\sum_{k\geq1}\chi_{D_\lambda}
(z)\log\frac{r}{|z|}d\V(z)\leq\frac{\alpha}{2(\alpha+\varepsilon)} \log\frac{1}{1-r^2}+O(1),\quad r\xrightarrow\ 1.
\end{equation*}
\end{lemma}

\begin{proof}
Let $X=\{(\lambda_k,m_{k})\}_{k\geq 1}$ be a zero divisor for $A^\infty_{\alpha}$ then there exists a non zero function $f\in\Bergmaninf$ such that
\begin{equation*}  f^{(k)}(\lambda)=0,\qquad 0\leq k< m_\lambda,\ \lambda\in\Lambda.
\end{equation*}

We will obtain our condition by redistributing the mass $\Delta(\log|f|)$ on the hyperbolic disks $D_{\lambda}$.

The function $\log|f|$ is subharmonic and not identically $-\infty$ and we have for all $z\in\mathbb{D}$ 
\begin{equation*}
\log|f(z)|\leq \log\|f\|_{\alpha,\infty}+\frac{\alpha}{2} \log\frac{1}{1-|z|^{2}}=:s(z).
\end{equation*} 
We use the same inductive method as in \cite{BHKM}. First we construct a new sub-harmonic function $h$ in $\mathbb{D}$ such that 
\begin{equation*}
\log|f(z)|\leq h(z) \leq s(z). 
\end{equation*}
Later, we will obtain our bound using Green's identity. Let $h_0=\log|f|$ and recall $r_{\lambda_k}=\sqrt{\frac{m_{k}}{m_{k}+\alpha+\varepsilon}}$, then 
\begin{equation*}
    h_0(z)=m_{1}\log\frac{|\varphi_{\lambda_{1}}(z)|}{r_{\lambda_1}}+ \underbrace{\log|f(z)|- m_{1}\log\frac{|\varphi_{\lambda_{1}}(z)|}{r_{\lambda_1}}}_{=U_0}.
\end{equation*}
Since $f/\varphi_{\lambda_1}^{m_{\lambda_1}}$ is holomorphic in $\U$,
$U_0$ is a subharmonic function on $\mathbb{D}$ and harmonic in a small neighborhood of $\lambda_{1}$. In fact, in a small neighborhood $V$ of $\lambda_{1}$ not
containing any other zero of $f$ than $\lambda_1$, we have
\begin{equation*}
    |f(z)|=|\varphi_{\lambda_1}(z)|^{m_1} |g(z)|, 
    \end{equation*} 
where $g$ is a holomorphic function with no zeros in $V$.

Therefore 
\begin{equation*}
    \Delta\log|f(z)|=m_1\Delta\log|\varphi_{\lambda_1}(z)|+\underbrace{\Delta\log|g|}_{=0},\quad z\in V.
\end{equation*}

Now we will modify the term $m_{1}\log\frac{|\varphi_{\lambda_{1}}(z)|}{   r_{\lambda_1}   }$ to obtain a constant function with respect to the invariant laplacian.
Let us consider
\begin{equation}
v_1(z)= \left\lbrace\begin{array}{lll} \frac{\alpha+\varepsilon}{2}\log\frac{1-r_{\lambda_1}^{2}}{1-|\varphi_{\lambda_{1}}(z)|^2}, ~z\in D(\lambda_1,r_{\lambda_1}), \\ \\ m_{1}\log\frac{|\varphi_{\lambda_{1}}(z)|}{r_{\lambda_1}}, ~~~~\mbox{ otherwise}.
\end{array}\right.
\end{equation}
Then $v_1$ is harmonic outside $D(\lambda_1, r_{\lambda_1})$ and $v_1\in\textit{C}^{1}(\mathbb{D})$. Using $\Delta=4\partial\overline{\partial}$, and the fact that
$\log|1-\overline{\lambda}z|$ is harmonic,  we have inside $D(\lambda_1, r_{\lambda_1})$
\begin{eqnarray*}
\Delta v_1
&=&\Delta\left(\frac{\alpha+\varepsilon}{2}\log\frac{1-r_{\lambda_1}^{2}}{1-|\varphi_{\lambda_{1}}(z)|^2}\right)= \frac{\alpha+\varepsilon}{2}\Delta\left(\log\frac{|1-\overline{\lambda_1}z|^2}{(1-|z|^2)(1-|\lambda_1|^2)}\right)\\
&=&(\alpha+\varepsilon) \Delta(\log|1-\overline{\lambda_1}z|)-\frac{\alpha+\varepsilon}{2}\Delta(\log(1-|z|^2))\\
&=& \frac{2(\alpha+\varepsilon)}{(1-|z|^2)^2}.
\end{eqnarray*}
(We have used that $\Delta \log(1-|z|^2)=\displaystyle -\frac{4}{(1-|z|^2)^2}$.)
Observe that with the definition of the invariant laplacian $\widetilde{\Delta}u=(1-|z|^2)^2\Delta u$, the preceding computation shows that $\widetilde{\Delta}v_1$ is constant.
We thus obtain the total mass of the measure $\Delta v_1$ on $D(\lambda_1, r_{\lambda_1})$ which is equal to $2(\alpha+\varepsilon) \V(D(\lambda_1,r_{\lambda_1}))$:
\begin{eqnarray}\label{rlambda}
    \int_{\mathbb{D}}\Delta v_1 dm
&=&\int_{D(\lambda_1,r_{\lambda_1})} 2(\alpha+\varepsilon)\frac{dm}{(1-|z|^2)^2}=
2(\alpha+\varepsilon)\V(D(\lambda_1,r_{\lambda_1}))\nonumber\\
&=&2(\alpha+\varepsilon)\frac{r_{\lambda_1}^2}{1-r_{\lambda_1}^2}.
\end{eqnarray}
By the specific definition of $r_{\lambda_1}$ this is equal to $2m_{\lambda_1}$.
\\

On the other hand, since $\Delta \log|z|=2\pi\delta_0$,
\begin{align*}
\int_{\mathbb{D}}\Delta\left(m_{1}\log\frac{|\varphi_{\lambda_{1}}(z)|}{r_{\lambda_1}}\right)dm(z)&=m_1\int_{\mathbb{D}}\Delta(\log|\lambda_1-z|)dm(z)\\
&= 2 m_1. 
\end{align*}
In particular, by the very definition of $r_{\lambda_1}$,
both total masses coincide so that in terms of Laplacians, we can
replace $m_1\log\frac{|\varphi_{\lambda_1}(.)|}{r_{\lambda_1}}$ by $v_1(.)$
in $D(\lambda_1,r_{\lambda_1})$. This yields the function $h_1=v_1+U_0$. Obviously,
\begin{equation*}
    h_1(z)=\frac{\alpha+\varepsilon}{2}\log\frac{1-r_{\lambda_1}^{2}}{1-|\varphi_{\lambda_{1}}(z)|^2}\chi_{D(\lambda_1,r_{\lambda_1})}(z)+m_{1}\log\frac{|\varphi_{\lambda_{1}}(z)|}{c_{\lambda_1}}\chi_{\mathbb{D}\backslash D(\lambda_1,r_{\lambda_1})}(z)+U_0.
\end{equation*}
Let us show that 
\begin{enumerate}[label=(\alph*)]
	\item $h_1(z)\leq s(z)$, $z\in\mathbb{D}$,
	\item $h_0(z)\leq h_1(z)$, $z\in\mathbb{D}.$
\end{enumerate}
We start proving $h_1\leq s$. This is clear for $z\not\in  D(\lambda_1,r_{\lambda_1})$, because $h_1=h_0$ in $\U\setminus D(\lambda_1,r_{\lambda_1})$. For $z\in D(\lambda_1,r_{\lambda_1})$, we consider the function
$w_1=v_1+U_0-s$, we have 
\begin{equation*}
\Delta w_1=\Delta\left(v_1+U_0-\frac{\alpha}{2}\log\frac{1}{1-|z|^2}\right)=\Delta U_0
+\Delta\frac{\varepsilon}{2}\log\frac{1}{1-|z|^2}\geq 0.
\end{equation*}
Hence $w_1$ is subharmonic on $D(\lambda_1,r_{\lambda_1})$, and for $\xi\in \d D(\lambda_1,r_{\lambda_1})$, since $v_1(\xi)=0$, we have 
\[
 w_1(\xi)=U_0(\xi)-s(\xi)=h_0(\xi)-s(\xi)=\log|f(\xi)|-s(\xi)\leq 0.
\]
So in the boundary $w_1$ is non-positive, so that it is non-positive throughout the disc by the maximum principle.

It remains to see $h_0\leq h_1$. Again outside the hyperbolic disc  $D(\lambda_1,r_{\lambda_1})$ we have $h_0= h_1$. In the disc $D(\lambda_1,r_{\lambda_1})$ we need to compare the following functions 
\begin{equation*}
\varphi(z)=\frac{\alpha+\varepsilon}{2}\log\frac{1-r_{\lambda_1}^{2}}{1-|\varphi_{\lambda_{1}}(z)|^2},
\end{equation*}
and 
\begin{equation*}
\psi(z)=m_{1}\log\frac{|\varphi_{\lambda_{1}}(z)|}{r_{\lambda_1}}.
\end{equation*}
More precisely, we have to show that $\psi\le \varphi$ on $D(\lambda_1,r_1)$.
For this, we use the auxiliary functions
$$\varphi :x\longmapsto \frac{\alpha+\varepsilon}{2}\log \frac{1-r_{\lambda_1}^2}{1-x^2},$$
and 
$$\psi: x\longmapsto m_1\log\frac{x}{r_{\lambda_1}}.$$
The function $\psi$ is concave, while $\varphi$ is convex and $\psi(r_{\lambda_1})=\varphi(r_{\lambda_1})=0$. Moreover 
$$\varphi{'}(r_{\lambda_1})=\psi{'}(r_{\lambda_1})=\sqrt{m_1(m_1+\alpha+\varepsilon)}$$
thus the two functions touch smoothly at $x=r_{\lambda_1}$, and $\psi\le \varphi$
on $(0,r_{\lambda_1}]$. As a consequence $h_0\leq h_1$ in $D(\lambda_1,r_{\lambda_1})$.

Now we construct $h_2$ in the same way as before. We have $h_1=v_1+U_0$, so we can write
\begin{align*}
h_1(z)&=m_{2}\log\frac{|\varphi_{\lambda_{2}}(z)|}{r_{\lambda_2}} +\underbrace{v_1+U_0-m_{2}\log\frac{|\varphi_{\lambda_{2}}(z)|}{r_{\lambda_2}}}_{U_1}\\
&= m_{2}\log\frac{|\varphi_{\lambda_{2}}(z)|}{r_{\lambda_2}}+U_1,
\end{align*}
where, since $f/(\varphi_{\lambda_1}^{m_1}\varphi_{\lambda_2}^{m_2})$ is
holomorphic, $U_1$ is a subharmonic function that is harmonic in a small neighborhood of $\lambda_2$. 
Again we modify the term $m_{2}\log\frac{|\varphi_{\lambda_{2}}(z)|}{r_{\lambda_2}}$ in the hyperbolic disc $D(\lambda_2,r_{\lambda_2})$, and set
\begin{equation*}
v_2(z)= \left\lbrace\begin{array}{lll} \frac{\alpha+\varepsilon}{2}\log\frac{1-r_{\lambda_2}^{2}}{1-|\varphi_{\lambda_{2}}(z)|^2}, ~z\in D(\lambda_2,r_{\lambda_2}) \\ \\ m_{2}\log\frac{|\varphi_{\lambda_{2}}(z)|}{r_{\lambda_2}}, ~~~~\mbox{ otherwise}.
\end{array}\right.
\end{equation*}
And as in the first step, set $h_2=v_2+U_1$.\\
  Iterating this procedure we obtain a sequence of sub-harmonic functions $(h_n)_n$, such that for every $z\in\mathbb{D}$, the sequence $(h_n(z))_n$ is increasing and 
$$\log|f(z)|\leq h_n(z)\leq s(z).$$
So the pointwise limit $h$ of the sequence $(h_n)_n$, which is still  subharmonic on $\mathbb{D}$, is comprised between $\log|f|$ and $s(z)$.

Observe that $h$ has been obtained from $\log|f|$ by replacing around each zero $\lambda$
of $f$ the function $m\log|\varphi_{\lambda}|/r_{\lambda}$ by $\frac{\alpha+\varepsilon}{2}
\log(1-r_{\lambda}^2)/(1-|\varphi_{\lambda}|^2)$ and by a harmonic function far from the zeros so that the laplacian of $\log|f|$ is given by the sum of the laplacians of $v_k$.
\\

Since $ h(z)\leq s(z)$, by Green's formula (\cite[Theorem 3.6, p. 59]{HKZ00}), for $0\leq r< 1$
\begin{align*}
\int_{D(0,r)}\Delta h(z)\log \frac{r}{|z|}dm(z)&=-2 h(0)+\frac{1}{r\pi}\int_{|z|=r}h(z) d|z|\\
&\leq 2\left(-h(0)+\log\|f\|_{\alpha,\infty}+\frac{\alpha}{2}\log\frac{1}{1-r^2}\right).
\end{align*}
On the other hand, 
\begin{equation*}
\int_{D(0,r)}\Delta h(z)\log \frac{r}{|z|}dm(z)\geq 2(\alpha+\varepsilon) \int_{D(0,r)} \sum_{k\geq 1}\chi_{D_{\lambda_k}}(z)\log \frac{r}{|z|}d\V(z).
\end{equation*}
Hence,
\begin{equation*}
\int_{D(0,r)} \sum_{k\geq 1}\chi_{D_{\lambda_k}}(z)\log \frac{r}{|z|}d\V(z)\leq \frac{\alpha}{2(\alpha+\varepsilon)}\log\frac{1}{1-r^2}+O(1), \qquad r\longrightarrow 1,
\end{equation*}
as required.
\end{proof}

We are now in a position to prove the uniqueness result.
\begin{theorem}\label{thmcovering}
Let $\alpha>0$, $\varepsilon>0$, and $X=\{(\lambda,m_\lambda)\}_{\lambda\in\Lambda}$ be a divisor such that
\begin{equation}\label{eqcovering}
    \bigcup_{\lambda\in\Lambda} D\left(\lambda,\sqrt{\frac{m_\lambda}{\alpha+\varepsilon+m_\lambda}}\right)=\U\setminus K
\end{equation}
for some compact $K\subset\U$.
Then $X$ is a uniqueness divisor for $\mathcal{A}^{\infty}_{\alpha}$.
\end{theorem}

Before giving the proof of this result, we mention that Theorem \ref{thmcovering2} easily
follows from this. Indeed, suppose $X$ {\it is} a zero divisor in $\mathcal{A}^2_{\alpha}
\subset \mathcal{A}^{\infty}_{\alpha+2}$, then \eqref{eqcovering} does not hold for any
compact $K$ where $\alpha$ is replaced by $\alpha+2$ as required
in Theorem \ref{thmcovering2}.

\begin{proof}
By contradiction, suppose $X$ is a zero divisor for $\mathcal{A}^{\infty}_{\alpha}$. Recall $ D_\lambda= D\left(\lambda,\sqrt{\frac{m_\lambda}{\alpha+\varepsilon+m_\lambda}}\right)$. By Lemma \ref{Lem2.1},
\begin{equation}\label{SumGrowth}
    c+\frac{\alpha}{2(\alpha+\varepsilon)}\log\frac{1}{1-r^2}\geq\int_{D(0,r)}\sum_{\lambda\in\Lambda}\chi_{ D_\lambda}(z)\log\frac{r}{|z|}\frac{dm(z)}{(1-|z|^2)^2}.
\end{equation}
Hence
\begin{align*}
 &   \int_{D(0,r)}\sum_{\lambda\in\Lambda}\chi_{ D_\lambda}(z)\log\frac{r}{|z|}d\V(z)
 =\int_{D(0,r)}\log\frac{r}{|z|}d\V(z)\\
  &  +\int_{D(0,r)\cap K} \left[\sum_{\lambda\in\Lambda} \chi_{ D_\lambda}(z)-1\right]\log\frac{r}{|z|}d\V(z)+\int_{D(0,r)\setminus K} \left[\sum_{\lambda\in\Lambda}\chi_{ D_\lambda}(z)-1\right]\log\frac{r}{|z|}d\V(z).
\end{align*}
Notice that
\begin{enumerate}[label=(\alph*)]
    \item when $z\in D(0,r)\setminus K$, then $\sum_{\lambda\in\Lambda}\chi_{ D_\lambda}(z)-1\ge 0$,
    \item  integration inside $K$ only contributes at most as an additive (negative) constant,
\end{enumerate}
so that
\[
  \int_{D(0,r)}\sum_{\lambda\in\Lambda}\chi_{ D_\lambda}(z)\log\frac{r}{|z|}d\V(z)
 \ge  \int_{D(0,r)}\log\frac{r}{|z|}d\V(z)+O(1)=\frac{1}{2}\log\frac{1}{1-r^2}+O(1).
\]
But this is in contradiction with \eqref{SumGrowth} which concludes the proof.
\end{proof}

We should stop here for a little observation. In the case of the Fock space, it was
enough to consider the critical radius in order to get the corresponding uniqueness result.
This was related to the observation that when we cover the whole plane by disks we will
necessarily encounter "big'' overlaps of these disks however far we are from
the origin, and that in this case $\log R/|z|$ can be arbitrarily big. The situation changes
in the unit disk where $\log R/|z|$ tends to zero, and much more subtle overlap conditions
have to be discussed. For the purpose of our discussions here it is sufficient to play on the 
parameter $\varepsilon$. For that reason we have to impose a covering condition with smaller radii involving $\varepsilon$ (equivalently the critical radius gives the uniqueness result in all weighted Bergman spaces $\mathcal{A}^{\infty}_{\alpha'}$ with weight $\alpha'>\alpha$). It is easy to check that for given $\varepsilon>0$ there exist $C_2>C_1>0$ (actually $C_2>\varepsilon>C_1>0$), such that for every $m_{\lambda}>C_2$,
\begin{equation}\label{CompCond}
 \frac{m_{\lambda}-C_2}{m_{\lambda}+\alpha}\le \frac{m_{\lambda}}{m_{\lambda}+\alpha+\varepsilon}\le \frac{m_{\lambda}-C_1}{m_{\lambda}+\alpha}.
\end{equation}

\bigskip

{ {We finish this section with the short argument leading to \eqref{Blaschke}.
Denoting by $D_e(u,r)=\{z\in\C:|z-u|<r\}$ a Euclidean disk, we have 
(see \cite[p.4]{Gar06})
\begin{eqnarray}
 D\left(\lambda,t \right)=D_{e} \left( \frac{ 1-t_\lambda^2 }{1-t_\lambda^2|\lambda|^2}\lambda,  \frac{ 1-|\lambda|^2 }{1-t_\lambda^2|\lambda|^2} t_{\lambda }\right)
\end{eqnarray}
Using the finite overlap, an area argument shows that
\begin{equation}\label{areaArgument}
\sum_{\lambda\in\Lambda} \left(\frac{ 1-|\lambda|^2 }{1-t_{\lambda }^2|\lambda|^2} t_{\lambda }\right)^2 \lesssim \pi^2,
\end{equation}
where 
\[
 t_{\lambda}=\sqrt{\frac{m_{\lambda}-C}{m_{\lambda}+\alpha}}.
\]
According to \eqref{CompCond}, playing on $C$, this is comparable to the same expression with
$\varepsilon$ in the denominator, but no constant subtracted in the numerator. This
will allow us to treat at the same time the situation with finite overlap (coming from
sampling), or the separation condition (coming from interpolation). The value of
$C$ will not be important in our estimates.

The finite overlap implies that the Euclidean radii $R_\lambda$ tend to zero when $\lambda$ approaches the boundary $\partial\mathbb{D}$:
\begin{equation}\label{smallradi}
R_\lambda=\frac{ 1-|\lambda|^2 }{1-t_{\lambda }^2|\lambda|^2} t_{\lambda }, \to 0, \quad\text{as} \quad |\lambda|\to 1.
\end{equation}
Since $t_{\lambda }=\sqrt{\frac{m_{\lambda}- C}{m_{\lambda}+\alpha}}\sim 1$, we have from \eqref{areaArgument}
 \[ \Big( \frac{ 1-|\lambda|^2}{1-t_{\lambda}^2|\lambda|^2} \Big)^{2} \to 0, \ \text{as} \   |\lambda|\to 1.\]
On the other hand 
 \begin{eqnarray*}
\frac{ 1-|\lambda|^2 }{1-t_{\lambda }^2|\lambda|^2}&=& (1-|\lambda|^2) \frac{1}{ 1-|\lambda|^2\big( \frac{m_\lambda-C}{m_\lambda+\alpha} \big) }     
=   (1-|\lambda|^2) \frac{1}{ 1-|\lambda|^2 \big(1- \frac{C+\alpha}{m_\lambda+\alpha} \big) }     
\\
&=&  \frac{1}{ 1+\big( \frac{(C+\alpha)|\lambda|^{2}}{(m_\lambda+1)(1-|\lambda|^2) } \big) } \to 0, \ \text{as} \   |\lambda|\to 1. \\
\end{eqnarray*}
Necessarily
 \[ (m_\lambda+1)(1-|\lambda|^2) \sim  m_\lambda(1-|\lambda|^2)  \to 0, \ \text{as} \  |\lambda|\to 1,\]
and thus
\[
 \frac{ 1-|\lambda|^2 }{1-t_{\lambda }^2|\lambda|^2}\simeq m_\lambda(1-|\lambda|^2),
\] 
which in view of \eqref{areaArgument} leads to \eqref{Blaschke}.}}

\section{Sampling  and interpolation in $\Bergman$}

\subsection{Local $L^2$-estimates}

We will obtain the results {\it via} a local control of the $\Bergman$-functions. Let us consider the $L^2$-norm in a domain $\Omega$, denoted by $\|\cdot\|_\Omega$. With this norm, we can obtain a control of the norm of the basis elements in terms of the regularized $\beta$-function.

For $\text{Re\;} a >0$ and $\text{Re\;} b >0$, we define the \textit{$\beta$-function} 
\begin{equation}\label{beta}
    \beta(a,b)=\int_0^1 t^{a-1}(1-t)^{b-1}dt.
\end{equation}
Recall that 
\begin{equation}\label{beta}
\beta(a,b)=\frac{\Gamma(a)\Gamma(b)}{\Gamma(a+b)},
\end{equation}
and  so by Stirling's formula (see \cite{AS}) 
\begin{equation}\label{beta_estimate}
\beta(n+1,\alpha+1) \sim\frac{\Gamma(\alpha+1)}{n^{1+\alpha}}, 
\end{equation}
The following sharp inequality is sometimes useful (see \cite{Wg}) 
\begin{equation}\label{GammaINEqua}
\Big( \frac{x}{x+s} \Big)^{1-s} \leq \frac{\Gamma(x+s)}{x^{s}\Gamma(x)}\leq 1 , \  x>0, \ 0<s<1.
\end{equation}
 For $x\in(0,1]$, we define the \textit{incomplete $\beta$-function}
\begin{equation}\label{incompletebeta}
    \beta(x;a,b)=\int_0^x t^{a-1}(1-t)^{b-1}dt,
\end{equation}
and the \textit{regularized incomplete $\beta$-function} (or \textit{regularized $\beta$-function} for short)
\begin{equation}\label{regularizedbeta}
    I(x;a,b)=\frac{\beta(x;a,b)}{\beta(a,b)}.
\end{equation}

With these notations in mind we can compute the norm of the orthonormal basis $\{e_j\}_{j}$
on smaller disks (we will of course be interested in disks of type $D_{\lambda}$). For
this we need the following result.
\begin{lemma}\label{L2.3}
Let $\{e_j\}_j$ be the orthonormal basis for $\mathcal{A}_\alpha^2$. Then, for all $j\geq 0$ and $0<r<1$
\begin{equation*}
    \|e_j\|_{D(0,r)}^2=(\alpha+1)I(r^2;j+1,\alpha+1).
\end{equation*}
\end{lemma}
\begin{proof}
Recall first from \eqref{lemmaorthonormal} that
\begin{equation*}
    e_n(z)=\sqrt{\frac{\Gamma(n+2+\alpha)}{n!\Gamma(2+\alpha)}}z^n,\qquad n\geq0.
\end{equation*}
Since $\Gamma(n)=(n-1)!$, and with \eqref{beta} in mind,
we can rewrite 
\[
 \beta(n+1,\alpha+1)=\frac{n!\Gamma(1+\alpha)}{\Gamma(n+2+\alpha)}=
 \frac{n!\Gamma(2+\alpha)}{(1+\alpha) \Gamma(n+2+\alpha)}, 
\]
and
\begin{equation}\label{p12*} 
    e_n(z)=\sqrt{\frac{1}{(1+\alpha)\beta(n+1,\alpha+1)}}z^n,\qquad n\geq0
\end{equation}
(which can be found also directly by computing the norm $\|z^n\|^2_{\alpha,2}$).

Now, the lemma follows  from  the following computation with the obvious 
substitution $s=t^2$,
\begin{eqnarray*}
    \|z^j\|_{D(0,r)}^2&=&(\alpha+1)\int_{D(0,r)}|z^j|^2(1-|z|^2)^{\alpha} dm(z) =2(\alpha+1)\int_0^r t^{2j+1}(1-t^2)^{\alpha}dt\\
    &=&(\alpha+1)\beta(r^2;j+1,\alpha+1).
\end{eqnarray*}
\end{proof}

For our later discussions we will thus need estimates on the regularized incomplete
$\beta$-function. 
{ 
\begin{lemma}\label{lemmabetaestimates}
Let $\alpha>-1$ and $r_m=\sqrt{\frac{m}{m+\alpha+1}}$ then
\begin{enumerate}[label=(\alph*)]
    \item For every $c>0$, there is $\varepsilon=\varepsilon(\alpha,c)>0$ such that for all $m\geq c+1$ and $n<m$
    \begin{equation*}
        I\left(\frac{m-c}{m+\alpha+1};n+1,\alpha+1\right)\geq \varepsilon.
    \end{equation*}
    \item For $t<r_m$, there exists $\varepsilon=\varepsilon(\alpha)$ such that
    \begin{equation*}
        F_{m,\alpha}(t)=\left[1-t^2\right]^{\alpha+2} \sum_{0\leq j<m}\frac{t^{2j}}{(\alpha+1)\beta(j+1,\alpha+1)}\geq\varepsilon.
    \end{equation*}
    \item Given $0<\eta<1$, there exists $a_\alpha(\eta)>0$ such that for all $j\geq m\geq a_\alpha(\eta)$
    \begin{equation*}
        I\left(\frac{m-a(\eta)}{m+\alpha+1};j+1,\alpha+1\right)\leq \eta I\left(\frac{m}{m+\alpha+1};j+1,\alpha+1\right).
    \end{equation*}
\end{enumerate}
\end{lemma}
}
The lemma does not appeal to a specific zero divisor. Later on, $m$ will correspond to
$m_{\lambda}$ and $r_m$ to $r_{{\lambda}}$. Compare the critical radius $r_m$ appearing
in this lemma with the one given in \eqref{CritValInfty} (as mentioned there, the critical value corresponds to the situation when $\varepsilon=0$): the term $\alpha$ appearing there turns into $\alpha+1$ here. 
\begin{proof}
(a) Since $\beta(x,n+1,\alpha+1)$ is increasing in $x$, for $m>n$
$$ I\left(\frac{m-c}{m+\alpha+1};n+1,\alpha+1\right)\geq I\left(\frac{n-c}{n+\alpha+1};n+1,\alpha+1\right).$$
We need to treat two cases.\\
{ 
\underline{First, suppose $\alpha>0$}, then } 
\begin{eqnarray*}
\beta\left(\frac{n-c}{n+\alpha+1};n+1,\alpha+1\right)&=&\int_{0}^{\frac{n-c}{n+\alpha+1}}t^{n}(1-t)^{\alpha}dt\\
&\geq &\Big(1- \frac{n-c}{n+\alpha+1}\Big)^\alpha \int_{0}^{\frac{n-c}{n+\alpha+1}}t^{n}dt\\\
&\geq & \Big(\frac{\alpha+1+c}{n+\alpha+1}\Big)^\alpha \frac{1}{n+1} \Big(1-\frac{\alpha+1+c}{n+\alpha+1}\Big)^{n+1},
 \end{eqnarray*}
which is comparable to $\displaystyle \frac{1}{(n+1)^{(\alpha+1)}}$ for $n\ge c+1$ 
(uniformly in $n$ and $m$). Now,
By \eqref{beta_estimate}, this yields inequality (a) at least when $n\ge c+1$.\\

\underline{Second, suppose $\alpha\in (-1,0]$}. Note that for $r\in (0,1), n\geq 0$ we have 
\begin{eqnarray}
    \int_0^{r}t^n(1-t)^\alpha dt&=& \frac{r^{n+1}}{n+1} (1- r)^{\alpha}  + \frac{\alpha}{n+1}
     \int_0^{r}t^{n+1}(1-t)^{\alpha-1} dt \nonumber\\
     &\geq &\frac{r^{n+1}}{n+1}(1-r)^\alpha+ \frac{\alpha}{n+1}\frac{r}{1-r} 
     \int_0^{r}t^{n}(1-t)^{\alpha} dt, \nonumber
\end{eqnarray}
where we have also used the fact that $\alpha r/(1-r)$ is decreasing in $r$ since $\alpha<0$.
Hence
\begin{equation}\label{A0}
  \int_0^{r}t^n(1-t)^\alpha dt \geq \frac{1}{1 +\frac{-\alpha}{n+1}\frac{r}{1-r}} \frac{r^{n+1}}{n+1}(1-r)^\alpha.
  \end{equation}
Let $r=\frac{n-c }{n+\alpha+1 }$ (the interesting case being $n\geq c+1$). In this situation we have 
\[ 1+\frac{-\alpha}{n+1}\times\frac{r}{1-r}= 1+\frac{-\alpha}{\alpha+1+c}\times\frac{n-c}{n+1}\leq 1+\frac{-\alpha}{\alpha+1+c}. \]
Thus, from \eqref{A0}
\begin{eqnarray*}
  \int_0^{ \frac{n-c }{n+\alpha+1 }  }t^n(1-t)^\alpha dt & \geq & \frac{1}{ [ 1+\frac{-\alpha}{\alpha+1+c}]} \times \frac{1}{(n+1)} \times \left( \frac{\alpha+1+c}{ n+\alpha+1}  \right)^{\alpha}\times  \left( \frac{n-c}{ n+\alpha+1}  \right)^{n+1}\\
    & \geq & \frac{ (\alpha+1+c)^{\alpha} e^{-(\alpha+1+c)(2+c)}}{[ 1+\frac{-\alpha}{\alpha+1+c}] } \times \frac{1}{\big(n+1\big)^{\alpha+1} }.    
\end{eqnarray*}
at the second inequality we have used $\frac{t-1}{t}\leq \log t , \ t>0$. 

On the other hand, by \eqref{GammaINEqua} we get
\begin{eqnarray*}
\frac{1}{\beta(n+1,\alpha+1)}& = &\frac{\Gamma(n+1+\alpha+1)}{\Gamma(\alpha+1) \Gamma(n+1)} 
\geq  \frac{\big( n+1\big)^{\alpha+1}}{\Gamma(\alpha+1) }.
\end{eqnarray*}
Put 
\[\varepsilon=\varepsilon(\alpha,c):=  \frac{1}{\Gamma(\alpha+1) } \times \frac{ (\alpha+1+c)^{\alpha} e^{-(\alpha+1+c)(2+c)}}{[ 1+\frac{-\alpha}{\alpha+2+c}] }. \]
Hence it follows from the inequalities above
\begin{equation}
I\left(\frac{n-c}{n+\alpha+2};n+1,\alpha+1\right)=\frac{  \int_0^{ \frac{n-c }{n+\alpha+2 }  }t^n(1-t)^\alpha dt }{\beta(n+1,\alpha+1)}\geq \varepsilon,
\end{equation}
which is the desired result.\\
 Finally, for $n<c+1$, independently whether $\alpha\geq 0$ or $\alpha\in (-1,0)$, the desired estimate follows from the fact that the integration interval $[0,\displaystyle \frac{m-c}{m+\alpha+1}]$ of the
incomplete $\beta$-function contains a fixed interval
$[0,\displaystyle \frac{1}{c+\alpha+2}]$ and both powers of $t$ and $(1-t)$ appearing
in the definition of $\beta(x;n+1,\alpha+1)$ are
controlled. Dividing by $\beta(n+1,\alpha+1)$ does not change this control since
$n$ is bounded.\\

(b)  Recall that the reproducing kernel of $\Bergman$,  $K(t,t)=(1-t^2)^{-\alpha-2}$  satisfies also
$$K(t,t)= \sum_{j\geq 0}\frac{t^{2j}}{(\alpha+1)\beta(j+1,\alpha+1)},\qquad 0\leq  t<1$$
Hence, it suffices to prove that 
 there exists $\epsilon=\epsilon(\alpha)$ and  $m_0$ such that for  $t<r_m$
$$R_{m,\alpha}(t)=(1-t^2)^{\alpha+2} \sum_{j\geq m}\frac{t^{2j}}{(\alpha+1)\beta(j+1,\alpha+1)} \le 1-\epsilon,\qquad m\geq m_0, $$
By  \eqref{beta_estimate}, 
\[
 \frac{1}{(\alpha+1)\beta(j+1,\alpha+1)}
 = \frac{1}{(\alpha+1)\Gamma(\alpha+1)/j^{\alpha+1}(1+o_j(1))}
 = \frac{j^{1+\alpha}}{\Gamma(2+\alpha)}(1+o_j(1)),
\]
where $o_j(1)$ tends to zero when $j$ tends to infinity.
So
\begin{eqnarray*}
R_{m,\alpha}(t)=\frac{(1-t^2)^{\alpha+2} }{\Gamma(\alpha+2)}\sum_{j\geq m} j^{1+\alpha}t^{2j}(1+ o_j(1))
\le\frac{(1-t^2)^{\alpha+2} }{\Gamma(\alpha+2)}(1+o_m(1))\sum_{j\geq m} j^{1+\alpha}t^{2j}.
\end{eqnarray*}
We will pass to an integral. For that, note that when $x\in [j-1,j]$
and $t\in [0,1)$, we have $t^{2j}\le t^{2x}$,
and obviously $j^{1+\alpha}=x^{1+\alpha}(1+o_j(1))$.
We deduce
\begin{eqnarray*}
R_{m,\alpha}(t)
&\le&(1+o_m(1))
 \frac{(1-t^2)^{\alpha+2} }{\Gamma(\alpha+2)}\int_{m-1}^{\infty} x^{1+\alpha}t^{2x}dx\\
&=&(1+o_m(1))\frac{(1-t^2)^{\alpha+2} }{\Gamma(\alpha+2)(\log 1/ t^2)^{2+\alpha}}\int_{(m-1)  (\log{1}/{t^2}) }^{\infty} u^{1+\alpha}e^{-u}du.
\end{eqnarray*}
Note that the function $t\longmapsto (1-t^2)/\ln (1/t^2)$ is bounded by $1$ on $(0,1)$.
Also, observe that $0< t<r_m=\dst \sqrt{\frac{m}{m+\alpha+1}}$, and hence
\[
 \log\frac{1}{t^2}>\log\frac{m+\alpha+1}{m}=\log(1+\frac{\alpha+1}{m+\alpha+1})
 =\frac{\alpha+1}{m+\alpha+1}(1+o_m(1)).
\] 
We deduce
\begin{eqnarray*}
R_{m,\alpha}(t)
&\leq& \frac{1+o_m(1)}{\Gamma(\alpha+2)}
 \int_{\frac{(\alpha+1)(m-1)}{m+\alpha+1}(1+o_m(1))}^{\infty} u^{1+\alpha}e^{-u}du.
\end{eqnarray*}
Let $\beta=\frac{(\alpha+1)(m-1)}{m+\alpha+1}(1+o_m(1))$, then using
the notation $\Gamma(a,b):=\int_b^\infty t^{a-1}e^{-t}dt$ for the incomplete Gamma function, we have
\[
 R_{m,\alpha}(t)\le \frac{1+o_m(1)}{\Gamma(\alpha+2)}\Gamma(\alpha+2,\beta).
\]
Now observe that for finite $m$ the estimate in (b) is trivially true. We can thus assume
that $m$ is big enough so that $\beta\ge \alpha+1-1=\alpha$.
So 
\[
 R_{m,\alpha}(t)\le \frac{1+o_m(1)}{\Gamma(\alpha+2)}\Gamma(\alpha+2,\alpha).
\]
Since $\alpha$ is fixed, we obviously have 
$q=\Gamma(\alpha+2,\alpha)/\Gamma(\alpha+2)<1$,
and for sufficiently big $m$ we have $1+o_m(1)<(1+q)/(2q)$ implying that
$R_{m,\alpha}\le (1+q)/2<1$, and we can pick $\varepsilon =(1-q)/2$.\\
\\  

(c) 
Clearly, setting
\[
 \rho=\frac{m-a(\eta)}{m+\alpha+1},\quad r=\frac{m}{m+\alpha+1},
\]
it is enough to show the estimate for the incomplete $\beta$-function:
\begin{eqnarray*}
\beta\left(\frac{m-a(\eta)}{m+\alpha+1};j+1,\alpha+1\right)
 &=&\int_0^{\rho}t^j(1-t)^{\alpha}dt\\
&\le& \eta \int_0^{r}t^j(1-t)^{\alpha}dt\\
&=&\eta \beta
\left(\frac{m}{m+\alpha+1};j+1,\alpha+1\right)
\end{eqnarray*}

\medskip

First \underline{assume $\alpha>0$}.

Note that
\begin{equation}\label{EqI}
    \beta_2:=\int_0^{r}t^j(1-t)^\alpha dt \geq (1-r)^\alpha  \frac{r^{j+1}}{j+1},
\end{equation}
and an integration by part, as well as the the fact that $t/(1-t)$ is increasing in $t$, yield
\begin{eqnarray*}
    \int_0^{\rho}t^j(1-t)^\alpha dt&=& \frac{\rho^{j+1}}{j+1} (1- \rho)^{\alpha}  + \frac{\alpha}{j+1}
     \int_0^{\rho}t^{j+1}(1-t)^{\alpha-1} dt \nonumber\\
     &\leq &\frac{\rho^{j+1}}{j+1}(1-\rho)^\alpha+ \frac{\alpha}{j+1}\frac{\rho}{1-\rho}
     \int_0^{\rho}t^{j}(1-t)^{\alpha} dt. 
\end{eqnarray*}
So,  if $\frac{\alpha}{j+1}\frac{\rho}{1-\rho}<1$, we get 
\begin{equation}\label{EqII}
 \beta_1:= \int_0^{\rho}t^j(1-t)^\alpha dt \leq \frac{1}{1 -\frac{\alpha}{j+1}\frac{\rho}{1-\rho}} \frac{\rho^{j+1}}{j+1}(1-\rho)^\alpha.
  \end{equation}
Given $0<\eta<1$ and $a>0$, 
it follows from \eqref{EqI} and \eqref{EqII}

\begin{eqnarray}\label{p16*}
\frac{\beta_1}{\beta_2}
\le \frac{1}{1 -\frac{\alpha}{j+1}\frac{\rho}{1-\rho}}\left(\frac{1-\rho}{1-r}\right)^{\alpha}
 \left(\frac{\rho}{r}\right)^{j+1}
\end{eqnarray}
Simple computations show
\[
 \frac{\rho}{r}=\frac{m-a(\eta)}{m}=1-\frac{a(\eta)}{m}<1,
\]
and
\[
 \frac{1-\rho}{1-r}=\frac{\alpha+1+a(\eta)}{\alpha+1} 
\quad 
\text{ and }\quad
 \frac{\rho}{1-\rho}=\frac{m-a(\eta)}{\alpha+1+a(\eta)}.
\]

Now for $j\ge m\ge a(\eta)$, 
\[
 \frac{\alpha}{j+1}\frac{\rho}{1-\rho} =
\frac{\alpha}{j+1}\frac{m-a(\eta)}{\alpha+1+a(\eta)}
 \le\frac{\alpha m}{(m+1)(\alpha+1)+(m+1)a(\eta)}\le \frac{\alpha}{\alpha+1+a(\eta)},
\]
therefore,
\[
\frac{1}{1-\frac{\alpha}{j+1}\frac{m-a(\eta)}{\alpha+1+a(\eta)}}
 \le \frac{\alpha+1+a(\eta)}{1+a(\eta)}.
\]

It remains the term
\begin{equation}\label{EqIII}
\left(\frac{\rho}{r}\right)^{j+1}= \left(1-\frac{a(\eta)}{m}\right)^{j+1}
 \le \left(1-\frac{a(\eta)}{m}\right)^{m}\le e^{-a(\eta)}
\end{equation}
Hence
\[
 \frac{\beta_1}{\beta_2}\le \frac{\alpha+1+a(\eta)}{1+a(\eta)}
\times \left(1+\frac{a(\eta)}{1+\alpha}\right)^{\alpha}
 \times e^{-a(\eta)}. 
\]
Since the exponential decrease dominates the polynomial growth, we have
\[
 \lim_{a\to+\infty}\frac{\alpha+1+a}{1+a}\times \Big(1+\frac{a}{1+\alpha}\Big)^{\alpha+1} \times e^{-a}=0,
\]
which proves the claim: for every $\eta>0$, we can find $a$, such that 
$\beta_1/\beta_2<\eta$ independently on $j\ge m\ge a$, and the same is true
for $I_1/I_2$.\\

\underline{ Second, assume $\alpha\in (-1,0]$}.
The exact same arguments as in the case $\alpha>0$ 
allow to reverse the
inequalities in \eqref{EqI} and \eqref{EqII} 
(since $\alpha$ is negative, the expression $\alpha r/(1-r)$ is decreasing  
and $(1-r)^{\alpha}$ is increasing in $r$). Hence, 

\begin{equation}\label{A}
  \beta_2:=\int_0^{r}t^j(1-t)^\alpha dt \geq \frac{1}{1 +\frac{-\alpha}{j+1}\frac{r}{1-r}} \frac{r^{j+1}}{j+1}(1-r)^\alpha,
  \end{equation}
and 
\begin{equation}\label{B}
    \beta_1:=\int_0^{\rho}t^j(1-t)^\alpha dt \leq  \ (1-\rho)^\alpha  \frac{\rho^{j+1}}{j+1}.
\end{equation}

As in \eqref{p16*} for $\alpha>0$, it follows from \eqref{A} and \eqref{B}
\begin{equation*}
\frac{\beta_1}{\beta_2}
\leq \Big(1 -\frac{\alpha}{j+1}\frac{r}{1-r}\Big)\left(\frac{1-\rho}{1-r}\right)^{\alpha}
 \left(\frac{\rho}{r}\right)^{j+1},
\end{equation*}
where the first factor is now in the numerator instead of the denominator.
Hence, with $\rho=\frac{m-a(\eta)}{m+\alpha+1} $ and $r=\frac{m}{m+\alpha+1}$ in mind, we get as before
\begin{equation*}
\frac{\beta_1}{\beta_2} \leq \left[   1 +\frac{-\alpha}{j+1} \times\frac{m}{\alpha+1} \right]
\times \left( 1+\frac{a(\eta)}{\alpha+1} \right)^{\alpha} \times 
 \left( 1-\frac{a(\eta)}{m}\right)^{j+1} .
\end{equation*}
Since $j+1>m$ and $\alpha<0$, we get 
\[
1 +\frac{-\alpha}{j+1} \times\frac{m}{\alpha+1}\leq 1 +\frac{-\alpha}{ \alpha+1}
 =\frac{1}{1+\alpha}. 
\]
Again
\[
 \left(1-\frac{a(\eta)}{m}\right)^{j+1}\leq e^{-a},
\]
and now
\[
 \left( 1+\frac{a(\eta)}{\alpha+1} \right)^{\alpha}<1,
\]
so that
\[
\frac{\beta_1}{\beta_2} \leq \frac{1}{1+\alpha} e^{-a},
\]
which again goes to 0, proving the claim also in this case.
 \end{proof}

As mentioned earlier, we will need to switch form the orthonormal basis $\{e_n\}_{n}$ in $\mathcal{A}^2_{\alpha}$ to an orthonormal basis on a smaller disk $D(0,r)$. The following lemma recalls this
simple fact.
\begin{lemma}\label{NormDiskr}
Let $f=\sum_{n\ge 0}a_ne_n\in \Bergman$. 
Then
\[
 (\alpha+1) 
\int_{D(0,r)}|f(z)|^2(1-|z|^2)^{\alpha}dm(z)
 =
\sum_{n\ge 0}I(r^2,n+1,\alpha+1)|a_n|^2, 
\] for every $r\in (0,1)$
\end{lemma}

\begin{proof}
Set
\[
 f_n=\frac{e_n}{\sqrt{I(r^2,n+1,\alpha+1)}},
\]
then, by Lemma (\ref{L2.3}), $(f_n)$ is orthonormal with respect to the measure
$dA_{\alpha}$ on $D(0,r)$. Writing now
\[
 f(z)=\sum_{n\ge 0} a_ne_n=\sum_{n\ge 0}(a_n \sqrt{I(r^2,n+1,\alpha+1)})f_n,
\]
we get the required equality.

\end{proof}

\begin{lemma}\label{lemmacontrolnorm}
Let $f\in\Bergman$, for every $c>0$ there exist  constants $A=A(c)>0$ such that for every $m\geq c+1$, we have 
\begin{equation*}
    \sum_{j<m}\left|\langle f,T_\lambda e_j\rangle\right|^2\leq A(\alpha+1)\int_{ D\left(\lambda,\sqrt{\frac{m-c}{m+\alpha+1}}\right)}|f(z)|^2(1-|z|^2)^\alpha dm(z),
\end{equation*}
for any $\lambda\in\U$.
\end{lemma}
\begin{proof}
Set $r=r_{m,\alpha,c}:=\sqrt{\frac{m-c}{m+\alpha+1}}$.
Taking $g=T_\lambda f$, the statement can be rewritten as
\begin{equation*}
    \sum_{0\leq j<m}\left|\langle g, e_j\rangle\right|^2 \lesssim(\alpha+1)\int_{ D(0,r)}
|g(z)|^2(1-|z|^2)^\alpha dm(z).
\end{equation*}
Since $g\in\Bergman$, 
\begin{equation*}
    g(z)=\sum_{j\geq0}a_j e_j=\sum_{j\geq0}a_j\frac{z^j}{\sqrt{(\alpha+1)\beta(j+1,\alpha+1)}}
\end{equation*}
and
\begin{equation*}
    \left|\langle g, e_j\rangle\right|=|a_j|^2, \ j\geq 0.
\end{equation*}
Using Lemma \ref{NormDiskr} applied to $g$, the claim is equivalent to
\begin{equation*}
    \sum_{0\leq j<m}|a_j|^2
\lesssim 
    \sum_{j\geq0}|a_j|^2 I\left(\frac{m-c}{m+\alpha+1};j+1,\alpha+1\right),
\end{equation*}
where we have substituted back the value of $r$.
Therefore it is enough to prove
\begin{equation*}
    I\left(\frac{m-c}{m+\alpha+1};j+1,\alpha+1\right)\geq\varepsilon(\alpha,c)>0
\end{equation*}
for $j<m$ and $m\geq c+1$, but this is given by  Lemma \ref{lemmabetaestimates}(a).
\end{proof}

The next lemma relates the finite overlap condition to a kind of Carleson measure type condition.

\begin{lemma}\label{lemmafoc}
	Let $X=\{(\lambda,m_\lambda)\}_{\lambda\in\Lambda}$ be a divisor. Then $X$ satisfies the finite overlap condition if and only if there exists a constant $C>0$ satisfying
	\begin{equation*}
	\sum_{\lambda\in\Lambda} \sum_{j<m_\lambda} |\langle f,T_\lambda e_j\rangle|^2\leq C\|f\|_{\alpha,2}^2,\quad f \in\Bergman.
	\end{equation*}
\end{lemma}
\begin{proof}
	Suppose that the estimate holds. Given $z\in\U$, set $f_{z}=T_z 1$, observe that $|T_z 1|=|k_z|$,
	so  $\|f_z\|_{\alpha,2}=1$ and
	\begin{align}
	|\langle f_{z}, T_\lambda e_j\rangle|^2&=|\langle T_z1,T_\lambda e_j\rangle|^2=|\langle e_0,T_zT_{\lambda}e_j\rangle|^2=|T_zT_{\lambda}e_j(0)|^2\nonumber\\
	&=\frac{1}{(\alpha+1)\beta(j+1,\alpha+1)}\left[1-|\varphi_\lambda(z)|^2\right]^{2+\alpha}|\varphi_\lambda(z)|^{2j} \label{eqTlambda},
	\end{align}
where we have used the form \eqref{p12*} of $e_j$.
	Then, by assumption,
	\begingroup
	\addtolength{\jot}{1em}
	\begin{align*}
	1=\|f\|^2_{\alpha,2}&\gtrsim\sum_{\lambda\in\Lambda}\sum_{j<m_{\lambda}}\frac{1}{(\alpha+1)\beta(j+1,\alpha+1)}\left[1-|\varphi_\lambda(z)|^{2}\right]^{2+\alpha}|\varphi_\lambda(z)|^{2j} \\
	&\geq \sum_{\lambda\in\Lambda}F_{m_\lambda,\alpha}(|\varphi_{\lambda}(z)|)\chi_{ D\left(\lambda,\sqrt{\frac{m_\lambda}{m_\lambda+\alpha+1}}\right)}(z)\geq \epsilon\sum_{\lambda\in\Lambda}\chi_{ D\left(\lambda,\sqrt{\frac{m_\lambda}{m_\lambda+\alpha+1}}\right)}(z),
	\end{align*}
	\endgroup
	where the function $F$ is defined in Lemma \ref{lemmabetaestimates}(b) and the last bound also comes from that result. Since the constants do not depend on $z$, we conclude that
	\begin{equation*}
	1\gtrsim\sum_{\lambda\in\Lambda}\chi_{ D\left(\lambda,\sqrt{\frac{m_\lambda}{m_\lambda+\alpha+1}}\right)}(z).
	\end{equation*}
  In the opposite direction, if $X$ satisfies the finite overlap condition, we just apply Lemma \ref{lemmacontrolnorm}, which gives
	\begin{align*}
		\sum_{\lambda\in\Lambda} \sum_{j<m_\lambda} |\langle f,T_\lambda e_j\rangle|^2&\lesssim \sum_{\lambda\in\Lambda} (\alpha+1)\int_{ D\left(\lambda,\sqrt{\frac{m_\lambda}{m_\lambda+\alpha+1}}\right)}  |f(z)|^2 (1-|z|^2)^\alpha dm(z)\\
		&=\sum_{\lambda\in\Lambda}(\alpha+1)\int_{\U} \chi_{ D\left(\lambda,\sqrt{\frac{m_\lambda}{m_\lambda+\alpha+1}}\right)}(z)|f(z)|^2(1-|z|^2)^\alpha dm(z)\\
		&\leq S_X\|f\|^2_{\alpha,2},
	\end{align*}
where $S_X$ is the overlap constant introduced in Definition (\ref{FiniteOvl}).

\end{proof}

\subsection{Sampling for $\mathcal{A}_\alpha^2$}

In order to obtain our geometric conditions we need the following local control of $\Bergman$-functions.
\begin{lemma}\label{lemmamenorradio}
	Given $0<\eta\leq 1$ there exists $a(\eta)>0$ such that if $f\in\Bergman$, and $m\geq a(\eta)$, and if
	\begin{align*}
	\sum_{j<m}|\langle f, e_j\rangle|^2&\leq\eta/2,\\
	(\alpha+1)\int_{ D\left(0,\sqrt{\frac{m}{m+\alpha+1}}\right)}|f(z)|^2&(1-|z|^2)^\alpha dm(z)\leq1,
	\end{align*}
	then
	\begin{equation*}
	(\alpha+1)\int_{ D\left(0,\sqrt{\frac{m-a(\eta)}{m+\alpha+1}}\right)}|f(z)|^2 (1-|z|^2)^\alpha dm(z)\leq\eta.
	\end{equation*}
\end{lemma}

As we will see from the proof, we have $a(\eta)=a_{\alpha}(\eta/2)$ where $a_{\alpha}$
appears in Lemma \ref{lemmabetaestimates}(c).

\begin{proof}
	If we write 
	\begin{equation*}
	f(z)=\sum_{j}a_j e_j(z),
	\end{equation*}
	by the first assumption and  \eqref{lemmaorthonormal}, using the orthogonality of $\{e_n\}_n$ with respect to the measure $dA_{\alpha}$ on any disk $D(0,r)$, $0<r<1$, and Lemma \ref{NormDiskr},
	\begin{eqnarray*}
\lefteqn{(\alpha+1)\int_{ D\left(0,\sqrt{\frac{m-a(\eta)}{m+\alpha+1}}\right)}|f(z)|^2 (1-|z|^2)^\alpha dm(z)}\\
	&=&\sum_{j\geq0}|a_j|^2 I\left({\frac{m-a(\eta)}{m+\alpha+1}};
 j+1,\alpha+1\right)\\
	&\le& \sum_{j< m}|a_j|^2 
 +\sum_{j\geq m}|a_j|^2 I\left({\frac{m-a(\eta)}{m+\alpha+1}};
 j+1,\alpha+1\right)\\
	&\le& \frac{\eta}{2}+\sum_{j\geq m} |a_j|^2 I\left(\frac{m-a(\eta)}{m+\alpha+1};j+1,\alpha+1\right).
	\end{eqnarray*}
	
Now, by Lemma \ref{lemmabetaestimates} (c) for $j\ge m \ge a_{\alpha}(\eta/2)$,
another application of Lemma \ref{NormDiskr} and the hypothesis,
	\begin{eqnarray*}
	\sum_{j\geq m} |a_j|^2 I\left(\frac{m-a(\eta/2)}{m+\alpha+1};j+1,\alpha+1\right)
	\leq\frac{\eta}{2}\sum_{j\geq m} |a_j|^2 
 I\left(\frac{m}{m+\alpha+1};j+1,\alpha+1\right)\\
	\leq(1+\alpha)\frac{\eta}{2}\int_{ D\left(0,\sqrt{\frac{m}{m+\alpha+1}}\right)}|f(z)|^2 (1-|z|^2)^\alpha dm(z)\leq \frac{\eta}{2},
	\end{eqnarray*}
	and the result follows.
\end{proof}
We are now in a position to prove Theorem \ref{ThmSamp2} which we restate here for convenience. 
\begin{theorem*} 
\begin{enumerate}[label=(\alph*)]
\item If $X$ is a sampling divisor for $\Bergman$, then $X$ satisfies the finite overlap condition  and there exists $0<C<\alpha+1$ such that
\begin{equation*}
\bigcup_{\lambda\in\Lambda}  D\left(\lambda,\sqrt{\frac{m_\lambda+C}{m_\lambda+\alpha+1}}\right)=\U.
\end{equation*}
\item Conversely, let the divisor $X$ satisfy the finite overlap condition and if there exists $C=C(S_X)>0$ such that for some compact $K$ of $\U$ we have
\begin{equation*}
\bigcup_{\lambda\in\Lambda,m_\lambda>C}  D\left(\lambda,\sqrt{\frac{m_\lambda-C}{m_\lambda+\alpha+1}}\right)=\U\setminus K,
\end{equation*}
then $X$ is a sampling divisor for $\Bergman$.
\end{enumerate}
\end{theorem*}
\begin{proof}
\textit{Necessary Part (a)}

Let $X=\{(\lambda,m_\lambda)\}_{\lambda\in\Lambda}$ be a sampling divisor. By Lemma \ref{lemmafoc}, it satisfies the finite overlap condition. Suppose that for every $0<C<\alpha+1$
\begin{equation*}
\bigcup_{\lambda\in\Lambda}  D\left(\lambda,\sqrt{\frac{m_\lambda+C}{m_\lambda+\alpha+1}}\right)\not=\U.
\end{equation*}
Then there exists a sequence $0<C_j\uparrow (\alpha+1)$ and  $z_j \in  \U$ such that
\begin{equation*}
z_j\notin\bigcup_{\lambda\in\Lambda}  D\left(\lambda,\sqrt{\frac{m_\lambda+C_j}{m_\lambda+\alpha+1}}\right). 
\end{equation*}
\\
Put 
 $   r_{\lambda, C_{j}}=\sqrt{\frac{m_{\lambda}+C_{j}}{m_{\lambda}+\alpha+1 }}$ and $r_{\lambda}=\sqrt{\frac{m_{\lambda}}{m_{\lambda}+\alpha+1 }}$. 
Thus
\begin{equation*}
z_{j}\in\mathbb{D}  \setminus \left[  \bigcup_{\lambda\in\Lambda}  D\left(\lambda,r_{\lambda, C_{j}}  \right)\right]  \subset 
\mathbb{D}\setminus \left[ \bigcup_{\lambda\in\Lambda}  D\left(\lambda, r_{\lambda}\right)   \right].
\end{equation*}
Let 
\[\zeta \in \bigcup_{\lambda\in\Lambda}  D\left(\lambda, r_{\lambda} \right)\]
 that is $\zeta\in D\big(\lambda_{0},r_{\lambda_0}\big)$ for some $\lambda_0\in\Lambda$ and we have for all $j$
  \[ \rho(\zeta,\lambda_{0})< r_{\lambda_{0}} ,\quad \text{ and } 
\quad \rho(z_j,\lambda_{0}  ) >  r_{\lambda_{0}, C_{j}} .\]
By the triangular  inequality for the pseudohyperbolic metric (see \cite[p.4]{Gar06}), we get 
\begin{eqnarray*}
\rho(z_j,\zeta)& \geq &\frac{\rho(z_j,\lambda_{0}  )-\rho(\zeta,\lambda_{0})} {1-\rho(z_j,\lambda_{0}) \rho(\zeta,\lambda_{0})} \geq \frac{  r_{\lambda_{0}, C_{j}} - r_{\lambda_{0}} } {1-  r_{\lambda_{0}, C_{j}} r_{\lambda_0}} \\
  & \geq &     \frac{ (  r_{\lambda_{0}, C_{j}})^{2} - (r_{\lambda_{0}})^{2}  } {1-( r_{\lambda_0} r_{\lambda_{0}, C_{j}} )^{2} }     
   \geq     \frac{ C_jm_{\lambda_0}+C_j(\alpha+1) }{ (2\alpha+2-C_j)m_{\lambda_0}+(\alpha+1)^2} \to 1 ,\qquad j\to \infty,
\end{eqnarray*}
and this latter convergence is uniform in $m$.
Since we  have chosen $\zeta$ arbitrary  in $ \bigcup_{\lambda\in\Lambda}  D\left(\lambda,r_\lambda \right)$, there exists a sequence $(z_j)_j\subset \mathbb{D}$ such that :
\[\rho_j:=\dist\Big(z_j, \bigcup_{\lambda\in\Lambda}  D\big(\lambda,r_{\lambda}\big)\Big) \rightarrow 1.\]

Set $f_j=T_{z_j}1$, observe that $|T_{z_j}1|=|k_{z_j}|$, and  applying  Lemma \ref{lemmacontrolnorm} we obtain
\begin{align*}
\sum_{\lambda\in\Lambda}\sum_{k<m_\lambda} |\langle f_j,T_\lambda e_k\rangle|^2&\lesssim\sum_{\lambda\in\Lambda}(\alpha+1)\int_{ D\left(\lambda,\sqrt{\frac{m_\lambda}{m_\lambda+\alpha+1}}\right)} \frac{(1-|z_j|^2)^{2+\alpha}}{|1-\overline{z_j}\zeta|^{2(2+\alpha)}}(1-|\zeta|^2)^\alpha dm(\zeta)\\
&\lesssim (\alpha+1)S_X\int_{\cup_{\lambda\in\Lambda} D\left(\lambda,\sqrt{\frac{m_\lambda}{m_\lambda+\alpha+1}}\right)}\left[1-|\varphi_{z_j}(\zeta)|^2\right]^\alpha  |\varphi_{z_j}(\zeta)|^2dm(\zeta)\\
&\lesssim(\alpha+1)\int_{|w|\geq\rho_j} (1-|w|^2)^\alpha dm(w)\rightarrow 0,
\end{align*}

since $\rho_j\to 1$. This contradicts the sampling inequality.\\

\textit{Sufficiency part (b)}

Suppose the divisor is not sampling. Then
there exists a sequence $(f_n)_{n\geq1}$ such that $\|f_n\|_{\alpha,2}=1$ and
\begin{equation*}
    \sum_{\lambda\in\Lambda}\sum_{0\leq j<m_\lambda}\left|\langle f_n,T_\lambda e_j\rangle\right|^2\rightarrow 0\quad\text{as }n\rightarrow \infty.
\end{equation*}
Passing to a weakly convergent sub-sequence denoted again by $(f_n)_{n\geq1}$ we have two possibilities: either (i) $(f_n)$ converges weakly to $f\not=0$ or (ii) $(f_n)$ converges weakly to $0$. \\
{ %
In the case (i), $X$ is a zero divisor for a function $f\in\Bergman$. 
By Theorem \ref{thmcovering2}, 
for every $\varepsilon>0$,
\[
 \U\setminus  \bigcup_{\lambda\in\Lambda} D\left(\lambda,\sqrt{\frac{m_\lambda}{\alpha+2+\varepsilon+m_\lambda}}\right)
\]

cannot be compact. Similarly as in \eqref{CompCond},
 given $\varepsilon>0$ for every $C_1>\varepsilon+1>0$ 
such that for every $m\ge 1$,
\[
 \frac{m-C_1}{m+\alpha+1}\le \frac{m}{m+\alpha+\varepsilon+2} 
\]

Hence
\[
 D\left(\lambda,\sqrt{\frac{m_\lambda-C_1}{ m_\lambda+\alpha+1}}\right)
\subset D\left(\lambda,\sqrt{\frac{m_\lambda}{m_\lambda+\alpha+2+\varepsilon }}\right),
\]
 for every $\lambda\in\Lambda$ with $ m_\lambda>C_2$. And
\[
 \U\setminus  \bigcup_{\lambda\in\Lambda, m_\lambda>C_1} D\left(\lambda,\sqrt{\frac{m_\lambda-C_1}{m_\lambda+\alpha+1}}\right)
\]
cannot be compact, leading to a contradiction.
}
\\

In the second case (ii), we define $\eta=\frac{1}{ S_X+1}$ and set the constant $C=a(\eta)$
where $a(\eta)$ is given by Lemma \ref{lemmamenorradio}. Denote
\begin{equation*}
    \Lambda_1=\left\{\lambda\in\Lambda: m_\lambda-C>0\right\}.
\end{equation*}

In order to reach a contradiction we will assume that the disks
$D\left(\lambda,\sqrt{\frac{m_\lambda-a(\eta)}{m_\lambda+\alpha+1}}\right)$,
$\lambda\in\Lambda_1$,
cover the unit disk up to a compact set, i.e.\ there is
$R=R(\eta)\in [0,1)$ such that
\begin{equation*}
    \U\setminus D(0,R)\subset\bigcup_{\lambda\in\Lambda_1} D\left(\lambda,\sqrt{\frac{m_\lambda-a(\eta)}{m_\lambda+\alpha+1}}\right).
\end{equation*}
We get for every $n\geq 1$
\[
    1=\int_\U |f_n(z)|^2dA_{\alpha}(z)
\leq\int_{D(0,R)} |f_n(z)|^2dA_{\alpha}(z)+
\sum_{\lambda\in\Lambda_1}\int_{ D\left(\lambda,\sqrt{\frac{m_\lambda-a(\eta)}{m_\lambda+\alpha+1}}\right)} |f_n(z)|^2 dA_{\alpha}(z).
\]
Denote by $\Lambda_2$ the set of $\lambda\in\Lambda_1$ such that
\begin{equation*}
    \sum_{j<m_\lambda} \left|\langle f_n,T_\lambda e_j\rangle\right|^2
\leq\frac{\eta}{2}(\alpha+1)\int_{ D\left(\lambda,\sqrt{\frac{m_\lambda}{m_\lambda+\alpha+1}}\right)}|f_n(z)|^2(1-|z|^2)^\alpha dm(z)=:\frac{\eta'}{2}.
\end{equation*}

By Lemma \ref{lemmamenorradio} applied for $\lambda\in\Lambda_2$, $\eta'$
and $f_n$ with $\|f_n\|_{\alpha,2}=1$, we obtain
\begin{align*}
    1\leq& (\alpha+1)\int_{D(0,R)} |f_n(z)|^2(1-|z|^2)^\alpha dm(z)\\
    &+\sum_{\lambda\in\Lambda_2}(\alpha+1)\int_{ D\left(\lambda,\sqrt{\frac{m_\lambda-a(\eta)}{m_\lambda+\alpha+1}}\right)} |f_n(z)|^2(1-|z|^2)^\alpha dm(z)\\
    \qquad&+\sum_{\lambda\in\Lambda_1\setminus\Lambda_2}(\alpha+1)\int_{ D\left(\lambda,\sqrt{\frac{m_\lambda-a(\eta)}{m_\lambda+\alpha+1}}\right)} |f_n(z)|^2(1-|z|^2)^\alpha dm(z)\\
    &\leq (\alpha+1)\int_{D(0,R)} |f_n(z)|^2(1-|z|^2)^\alpha dm(z)\\
    &+\eta\sum_{\lambda\in\Lambda_2}(\alpha+1)\int_{ D\left(\lambda,\sqrt{\frac{m_\lambda}{m_\lambda+\alpha+1}}\right)} |f_n(z)|^2(1-|z|^2)^\alpha dm(z)\\
    &+\sum_{\lambda\in\Lambda_1\setminus\Lambda_2}(\alpha+1)\int_{ D\left(\lambda,\sqrt{\frac{m_\lambda-a(\eta)}{m_\lambda+\alpha+1}}\right)} |f_n(z)|^2(1-|z|^2)^\alpha dm(z).
\end{align*}
For the first term on the right hand side of the last inequality, 
 the weak convergence, {\it via} dominated convergence, yields  
\begin{align*}
    (\alpha+1)\int_{D(0,R)} &|f_n(z)|^2(1-|z|^2)^\alpha dm(z) = o(1).\quad(n\rightarrow\infty).
\end{align*}

And by the definition of $\Lambda_2$
\begin{align*}
    &\sum_{\lambda\in\Lambda_1\setminus\Lambda_2}(\alpha+1)\int_{ D\left(\lambda,\sqrt{\frac{m_\lambda-a(\eta)}{m_\lambda+\alpha+1}}\right)} |f_n(z)|^2(1-|z|^2)^\alpha dm(z)\\
    &\leq \sum_{\lambda\in\Lambda_1\setminus\Lambda_2}(\alpha+1)\int_{ D\left(\lambda,\sqrt{\frac{m_\lambda}{m_\lambda+\alpha+1}}\right)} |f_n(z)|^2(1-|z|^2)^\alpha dm(z)\\
    &\leq\frac{2}{\eta}\sum_{\lambda\in\Lambda_1\setminus\Lambda_2}\sum_{j<m_\lambda} \left|\langle f_n,T_\lambda e_j\rangle\right|^2=o(1)\quad(n\rightarrow\infty)
\end{align*}
Finally,
\begin{align*}
    1&\leq o(1)+\eta \sum_{\lambda\in\Lambda}(\alpha+1)\int_{ D\left(\lambda,\sqrt{\frac{m_\lambda}{m_\lambda+\alpha+1}}\right)} |f_n(z)|^2(1-|z|^2)^\alpha dm(z)\\
    &\leq o(1)+\eta S_X (\alpha+1)\int_\U|f_n(z)|^2(1-|z|^2)^\alpha dm(z)=o(1)+\frac{S_X}{S_X+1},\quad n\rightarrow\infty.
\end{align*}
We have reached a contradiction.

\end{proof}

\subsection{Interpolation for $\mathcal{A}_\alpha^2$}
As in \cite{BHKM}, we obtain the geometric condition by a $\dbar$-scheme and a local control of the functions in the space. To do so, we adapt the same technique as in \cite{BO95}. We will need the following version of H\"{o}rmander's $L^2$-estimates for $\dbar$ due to Ohsawa \cite{Oh}. Recall the definition of the invariant laplacian
\begin{equation*}
    \widetilde{\Delta}=(1-|z|^2)^2\Delta
\end{equation*}
and define the invariant convolution of two functions $f$, $g$
\begin{equation*}
    (f\star g)(z)=\int_\U g(\varphi_z(\zeta))f(\zeta)d\V(\zeta).
\end{equation*}

\begin{theorem}{\bf{(Ohsawa)\cite{Oh}}}\label{ThmOhsawa}
Let $\psi$ be any subharmonic function in the disk such that $\widetilde{\Delta}\psi>\delta>0$. Then there is a solution $u$ to the equation $\dbar{u}=g$ such that
\begin{equation}\label{EstimOhsawa}
    \int_\U |u(z)|^2\frac{e^{-\psi(z)}}{1-|z|^2}dm(z)\leq C_\delta\int_\U|g(z)|^2e^{-\psi(z)}(1-|z|^2)dm(z).
\end{equation}
\end{theorem}

  We need to regularize the weight in such a way that we will not destroy the interpolation after the $\dbar$-surgery. This will be achieved by Lemma \ref{TheWeightedW} below.
We recall from Lemma \ref{lemmabetaestimates} the critical radius in $\mathcal{A}^2_{\alpha}$,
$r_\lambda=\sqrt{\frac{m_\lambda}{m_\lambda+\alpha+1}}$, 
and its dilation from Theorem \ref{ThmInter2} $r_\lambda'=\sqrt{\frac{m_\lambda+C_X}{m_\lambda+\alpha+1}}$, 
and denote  the associated hyperbolic  disks
$ D_\lambda= D\left(\lambda, r_\lambda\right)$ and $ D_\lambda'= D\left(\lambda,r_\lambda'\right)$ respectively.

We will also need the following auxiliary function (see \cite[p.119]{BO95})
\begin{equation*}
    \xi(\zeta)=  \xi_{\lambda}(\zeta)=\left\{
     \begin{array}{@{}l@{\thinspace}l}
       0,  &\quad 0\leq  |\zeta|<r_{\lambda}, \\
       \frac{\log\frac{1}{|\zeta|^2}}{K(m_\lambda,c,\alpha)}, &\quad r_{\lambda}<|\zeta|<r_\lambda', \\
       0, &\quad r_\lambda'<|\zeta|<1, \\
     \end{array}
   \right.
\end{equation*}
where
\begin{equation*}
    K=K(m_\lambda,C_X,\alpha)=\int_{r_{\lambda}<|\zeta|<r_\lambda'}\log\frac{1}{|\zeta|^2}d\V(\zeta).
\end{equation*}
(so that the $L^1$-norm of $\xi$ is one, for some more precise estimates on $K$ see below). 
 Consider the weight
\begin{eqnarray}\label{p23*}
    w_{\Lambda,\alpha}(z):&=&\sum_{\lambda\in\Lambda}m_\lambda\left[\log|\varphi_\lambda(z)|^2-\frac{1}{K}\int_{r_{\lambda}<|\zeta|<r_\lambda'}\log|\varphi_{\varphi_\lambda(z)}(\zeta)|^2\log\frac{1}{|\zeta|^2}d\V(\zeta)\right]\chi_{D_\lambda'}(z)\nonumber\\
&=&\sum_{\lambda\in\Lambda}m_\lambda\left[\log|\varphi_\lambda(z)|^2-\int_{\U}\log|\varphi_{\varphi_\lambda(z)}(\zeta)|^2\xi(\zeta)d\V(\zeta)\right]\chi_{D_\lambda'}(z)
\end{eqnarray}

\begin{lemma}\label{TheWeightedW} 
Let $X=\{(\lambda,m_\lambda)\}_{\lambda\in\Lambda}$ be a divisor and let $C_X$ be such that 
$(\alpha+1)(1-e^{-1})<C_X<\alpha+1$ the hyperbolic disks  $$ \left\{ D\left(\lambda,\sqrt{\frac{m_\lambda+C_X}{m_\lambda+\alpha+1}}\right)\right\}_{\lambda\in\Lambda}$$ are pairwise disjoint. 
Then the weight $ w_{\Lambda,\alpha}$ above satisfies
\begin{enumerate}[label=(\alph*)]
    \item $w_{\Lambda,\alpha}\leq0$,
    \item $-w_{\Lambda,\alpha}\leq A(\alpha)$ in $ D_\lambda'\setminus D_{\lambda}$,
    \item and $\widetilde{\Delta}w_{\Lambda,\alpha}\geq -4(\alpha+1-\varepsilon)$
for some $\varepsilon$ depending on $C_X$.
\end{enumerate}
\end{lemma}

\begin{proof}
To see $(a)$, since $\log|\varphi_a|$ is sub-harmonic, we have
\begin{equation*}
    \log|a| \times \log\frac{1}{r^2}\leq\frac{1}{2\pi}\int_0^{2\pi} \log|\varphi_a(re^{i\theta}))|\log\frac{1}{r^2}d\theta.
\end{equation*}
Integrating from $r_{\lambda}$ to $r_\lambda'$ with respect to the measure $\frac{rdr}{(1-r^2)^2}$ and dividing by $K$ yields the required result.

For $(b)$, observe that the separation condition implies that it is sufficient to consider
only one term of the sum. Let us also set $a=\varphi_{\lambda}(z)$, and notice that
$z\in D_{\lambda'}\setminus D_{\lambda}$ implies in particular $|a|=|\varphi_{\lambda}(z)|>
r_{\lambda}$. Hence
\begin{eqnarray*}
-w_{\Lambda,\alpha}(z)&=&
m_\lambda\left[\frac{1}{K(m_\lambda,C_X,\alpha)}\int_{r_\lambda<|\zeta|<r_\lambda'}
 \underbrace{\log|\varphi_{a}(\zeta)|^2\log\frac{1}{|\zeta|^2}}_{<0}d\V(\zeta)-\log|a|^2\right]\\
 &\le& m_{\lambda}\log\frac{1}{|a|^2}<m_{\lambda}\log\frac{1}{r_{\lambda}^2}
 =m_\lambda\log\frac{m_\lambda+\alpha+1}{m_\lambda}\\
 &\le &{\alpha +1}.
\end{eqnarray*}

Let us discuss $(c)$. Setting $a=\varphi_{\lambda}(z)$, we have
\[
 h(z):=\int_{\U}\log|\varphi_{\varphi_\lambda(z)}(\zeta)|^2\xi(\zeta)d\V(\zeta)
=\int_{\U}\log|\varphi_{a}(\zeta)|^2\xi(\zeta)d\V(\zeta)
 =(\xi\star E)(a),
\]
with $E(u)=\log|u|^2$.
Observe that $\widetilde{\Delta}(h\circ\varphi_{\lambda})=(\widetilde{\Delta}h)\circ\varphi_{\lambda}$
(see e.g.\ \cite[p.120]{BO95}), $\widetilde{\Delta}(\mu\star\log|\cdot|^2)=4\mu$ for any measure (notice that in \cite{BO95}, the authors define $\Delta=\partial
\overline{\partial}$ so that in our setting we have to introduce an additional factor 4), 
and hence
\[
\widetilde{\Delta}h=\widetilde{\Delta}[(\xi\star E)\circ\varphi_{\lambda}]= 4 \times (\xi\circ\varphi_{\lambda}).
\]
Hence, for $z\in D_\lambda'\setminus D_\lambda$
\begin{align*}
    \widetilde{\Delta}w_{\Lambda,\alpha}(z)&=m_\lambda\left(\underbrace{\widetilde{\Delta}\log|\varphi_\lambda(z)|^2}_{4\pi (1-|z|^2)\delta_{\lambda}}-\widetilde{\Delta}h(z)\right)\\
    &\geq -4 m_\lambda
 (\xi \circ\varphi_{\lambda}) (z).
\end{align*}

Since $\xi$ is decreasing, we get
\[
\widetilde{\Delta}w_{\Lambda,\alpha}(z)\ge -4m_{\lambda}\xi(r_{\lambda})=
-4\frac{m_{\lambda}}{K}\log \frac{m_{\lambda}+\alpha+1}{m_{\lambda}}
\ge -4\frac{\alpha+1}{K}.
\]
Let us estimate
\[
 K=2\int_{r_{\lambda}}^{r_{\lambda}'}\frac{-r\ln r^2}{(1-r^2)^2}dr
\]

The function $h(r)= \dfrac{ r^{2} }{1-r^{2} }\log\dfrac{1}{r^2}$ is increasing on $(0,1)$,  we get 
 \begin{align}\label{p24*}
 K
  &= \int_{(r_{\lambda})^2}^{(r_{\lambda}')^2}\log\frac{1}{t}\frac{ dt}{(1-t)^2} =  \left[ \frac{ 1}{1-t} \log\frac{1}{t} \right]_{r_{\lambda}^2 }^{r_{\lambda}'^2} + \int_{r_{\lambda}^2}^{r_{\lambda}'^2}\frac{1}{t(1-t)}dt \nonumber \\
  &=  h(r_{\lambda}') - h(r_{\lambda}) +   \log\frac{ 1-r_{\lambda}^2}{1- r_{\lambda}'^2 }  \nonumber\\    
&\geq   \log\frac{ \alpha+1 }{ \alpha+1-C_X }
\end{align}
Since $C_X>(1+\alpha)(1-e^{-1})$, we have $K>1$ as required. 
\end{proof}

We are ready to establish our conditions for interpolating divisors.
We recall the statement of the corresponding Theorem \ref{ThmInter2} here for the convenience of 
the reader. 
\begin{theorem*}
Let $\alpha>-1$.
\begin{enumerate}[label=(\alph*)]
\item 
If $X$ is an interpolating divisor for $\Bergman$, then there exists $C_X>0$ such that the hyperbolic disks $$\left\{ D\left(\lambda,\sqrt{\frac{m_\lambda-C_X}{m_\lambda+\alpha+1}}\right)\right\}_{\lambda\in\Lambda, m_\lambda>C_X}$$ are pairwise disjoint.
\item 
Conversely, if for some  $C_X$ such that   $(\alpha+1)(1-e^{-1})<C_X<\alpha+1$,  the hyperbolic disks  $$ \left\{ D\left(\lambda,\sqrt{\frac{m_\lambda+C_X}{m_\lambda+\alpha+1}}\right)\right\}_{\lambda\in\Lambda}$$ are pairwise disjoint, then $X$ is an  interpolating divisor for $\Bergman$.
\end{enumerate}
\end{theorem*}
\begin{proof}
{\it Sufficiency part.}

The proof is based on a $\overline{\partial}$-method which consists, as usual, in constructing first a smooth interpolating function, and then to use H\"{o}rmander's solution to the $\dbar$-equation with $L^2$-estimates to make the interpolating function holomorphic without destroying the interpolation.

Given $v=(v_\lambda^j)_{\lambda\in\Lambda,\, \alpha<m_\lambda}\in \ell^2(X)$, take polynomials $p_\lambda$, $\lambda\in\Lambda$, with $\deg p_{\lambda}\le m_{\lambda}-1$, such that
\begin{equation*}
\langle p_\lambda,e_j\rangle=v_\lambda^j,\quad \lambda\in\Lambda,\, j<m_\lambda.
\end{equation*}
We recall that the above interpolation condition means that we interpolate germs in $\lambda$
and that it is thus sufficient to guarantee that the interpolating function and its derivatives take the values $v_{\lambda}^j$ in $\lambda$, $0\le j\le m_{\lambda}-1$.  
Recall from \eqref{p5*} that $N^{2,\alpha}_{\lambda,m_{\lambda}}$ denotes the set of functions
$f$ in $\mathcal{A}^2_{\alpha}$ vanishing up to the order $m_{\lambda}-1$ in $\lambda$. 
Since $v\in \ell^2(X)$, we have
\begin{equation*}
\|p_\lambda\|_{\Bergman/N^{2,\alpha}_{0,m_\lambda}}=\|p_\lambda\|_{\alpha,2},\quad \sum_{\lambda\in\Lambda}\|p_\lambda\|_{\Bergman/N^{2,\alpha}_{0,m_\lambda}}^2=\|v\|_2^2<\infty.
\end{equation*}

Let us denote $Q_\lambda=T_\lambda p_\lambda$, $r_\lambda=\sqrt{\frac{m_\lambda}{m_\lambda+\alpha+1}}$, $r_\lambda'=\sqrt{\frac{m_\lambda+C_X}{m_\lambda+\alpha+1}}$, $ D_\lambda= D\left(\lambda,r_\lambda\right)$ and $ D_\lambda'= D\left(\lambda,r_\lambda'\right)$. Notice that $ \{D_\lambda'\}_{\lambda\in\Lambda}$ are pairwise disjoint by hypothesis. Consider the smooth interpolating function
\begin{equation}
F(z)=\sum_{\lambda\in\Lambda} Q_\lambda(z)\eta\left(|\varphi_\lambda(z)|-r_\lambda'\right),
\end{equation}
where $\eta=\eta_\lambda$ is a smooth cut-off function on $\R$, so that
\begin{enumerate}[label=(\alph*)]
	\item $\text{supp }\eta\subset(-\infty,0],$\label{etaprop1}
	\item $\eta\equiv 1$ on $(-\infty,r_\lambda-r'_\lambda],$ \label{etaprop2}
	\item $|\eta'|\lesssim \frac{1}{r_\lambda'-r_\lambda}\simeq 2\frac{m_\lambda+\alpha+1}{C_X}$. \label{etaprop3}
\end{enumerate}

The separation condition implies that in each $z$, $F(z)$ is given by at most one term.

Notice that $\text{supp } F\subset \bigcup_{\lambda\in\Lambda} D_\lambda'$. Also
\begin{equation*}
\text{supp }\dbar F\subset \bigcup_{\lambda\in\Lambda}\left( D_\lambda'\setminus  D_\lambda\right),
\end{equation*}
since $\eta\left(|\varphi_\lambda(z)|-r_\lambda'\right)$ is constant outside
$\bigcup_{\lambda\in\Lambda}\left( D_\lambda'\setminus  D_\lambda\right)$.

A direct calculation shows that
\begin{equation}\label{p25*}
   | \overline{\partial} \eta\left(|\varphi_\lambda(z)|-r_\lambda'\right)|  \leq \frac{1}{2}  \|\eta'\|_{\infty}   \frac{1-|\varphi_\lambda(z)|^2}{1-|z|^2},
\end{equation}
so that for $z\in  D'_\lambda\setminus  D_\lambda$, property \ref{etaprop3} yields
\begin{eqnarray}\label{p26*}
|\dbar F(z)|&\leq& |Q_\lambda(z)|\times |\dbar \eta|\times |\dbar\big(|\varphi_\lambda(z)|\big)|\nonumber\\
&\lesssim& \frac{|Q_\lambda(z)|}{r_\lambda'-r_\lambda}   \frac{1-r_\lambda^2}{1-|z|^2} .  
\end{eqnarray}

Since $T_\lambda$ is an isometry of $\Bergman$,  $$\|Q_\lambda\|_{\alpha,2}^2 =\|p_\lambda \|_{\alpha,2}^2=
\|p_\lambda\|^2_{\Bergman/N_{0,m_\lambda}^2},$$ and therefore $F$ has the growth of $\Bergman$:
\begin{align*}
(\alpha+1)\int_{\U} |F(z)|^2(1-|z|^2)^\alpha dm(z)&\le\sum_{\lambda\in\Lambda}(\alpha+1)\int_{ D_\lambda'} |Q_\lambda(z)|^2 (1-|z|^2)^\alpha dm(z)\\ &\leq\sum_{\lambda\in\Lambda}\|Q_\lambda \|_{\Bergman}^2=\|v\|_2^2.
\end{align*}

Next we construct a holomorphic interpolating function using a $\dbar$-technique. As in the scheme used in \cite{BO95}, we are looking for a holomorphic interpolating function of the from $f=F-u$, where $u$ is a solution to the $\dbar$-problem $\dbar u=\dbar F$ with the conditions
\begin{equation*}
\int_{\U} |u(z)|^2 (1-|z|^2)^\alpha dm(z)<\infty,
\end{equation*}
and
\begin{equation*}
\d^j u(\lambda)=0,\ \forall j<m_\lambda.
\end{equation*}
This last condition will ensure that 
$$\partial^j f(\lambda)=\partial^j F(\lambda)
,\qquad  j<m_\lambda,$$
and we remind that the interpolation condition $\langle f,T_{\lambda}e_j\rangle=v_{\lambda}^j$
translates into an interpolation by germs. 

We will apply Ohsawa's Theorem \ref{ThmOhsawa} with the subharmonic weight
$$\phi(z)=(\alpha+1)\log\frac{1}{1-|z|^2}+w_{\Lambda,\alpha}(z),$$ where $w_{\Lambda,\alpha}$ is the weight in Lemma \ref{TheWeightedW}. 
We need to compute $\widetilde{\Delta}\phi$ as in Ohsawa's theorem:
{ %
\begin{align*}
 \widetilde{\Delta}\phi&=(1-|z|^2)^2\Delta \left((\alpha+1)\log\frac{1}{1-|z|^2}+w_{\Lambda,\alpha}\right)\\
&=4(\alpha+1)+(1-|z|^2)^2\Delta w_{\Lambda,\alpha}\\
&\geq 4(\alpha+1)- 4(\frac{\alpha+1}{K})>\varepsilon. 
\end{align*}
The last inequality is due to Lemma \ref{TheWeightedW} (c). 
}\\
The properties of the weight $w_{\Lambda,\alpha}$ and Ohsawa's estimate \eqref{EstimOhsawa} yield
\begin{eqnarray*}
\lefteqn{\int_{\U}|u(z)|^2(1-|z|^2)^\alpha dA_\alpha(z)=(\alpha+1)\int_{\U}|u(z)|^2\frac{e^{-(\alpha+1)\log\frac{1}{1-|z|^2}}}{1-|z|^2}\frac{dm(z)}{\pi}}\\
&&\leq(\alpha+1)\int_{\U}|u(z)|^2\frac{e^{-\phi(z)}}{1-|z|^2}\frac{dm(z)}{\pi} \quad \text{(Lemma \ref{TheWeightedW}(a))}\\
&&\lesssim \int_{\U}|\dbar F(z)|^2e^{-\phi(z)}(1-|z|^2)dm(z) \quad\text{(Ohsawa)}\\
&&\lesssim\sum_{\lambda\in\Lambda}\int_{ D_\lambda'\setminus  D_\lambda}\frac{|Q_\lambda(z)|^2}{(r_\lambda'-r_\lambda)^2}\Big( \frac{1-r_\lambda^2}{1-|z|^2}\Big)^2(1-|z|^2)^{\alpha+2} dm(z)\quad \text{(Lemma \ref{TheWeightedW}(b) \& \eqref{p26*})}
\end{eqnarray*}
Now
\begin{equation}\label{p27*}
\frac{1-r_{\lambda}^2}{r_\lambda'-r_\lambda}=\frac{(r_\lambda'+r_\lambda)(1-r_{\lambda}^2)}{r_\lambda'^{2}-r_\lambda^{2}}\leq 2 \frac{\alpha+1}{C_{X}}, \qquad \lambda\in\Lambda,
\end{equation}
so that
\[
\int_{\U}|u(z)|^2(1-|z|^2)^\alpha dA_\alpha(z)
\lesssim\sum_{\lambda\in\Lambda}\|Q_\lambda\|^2_{\Bergman/N^2_{0,m_\lambda}}<\infty.
\]
Hence, $f=F-u\in \Bergman$.

Finally, we want to see that $\langle f,T_\lambda e_j \rangle=v_\lambda^j$, $j<m_\lambda$.
We have already mentioned that for this we need $u$ to vanish at order $m_\lambda$ in each $\lambda$, so let us examine the order of the singularity near $\lambda$. For each $z\in  D_\lambda$
\begin{equation*}
	w(z)= m_\lambda \log|\varphi_\lambda(z)|^2+C_{\lambda}.
\end{equation*}
and therefore
\begin{align*}
	+\infty&>\int_{ D_\lambda}|u(z)|^2e^{-w(z)}dm(z)\gtrsim \int_{ D_\lambda}|u(z)|^2 e^{-\log|\varphi_\lambda(z)|^{2m_\lambda}}dm(z)\\
	&=\int_{ D_\lambda}|u(z)|^2\frac{1}{|\varphi_\lambda(z)|^{2m_\lambda}}dm(z).
\end{align*}
This forces $u$ to vanish at order $m_\lambda$ on $\lambda\in\Lambda$. Therefore
\begin{equation*}
\langle f,T_\lambda e_j\rangle=v_\lambda^j,\qquad j\leq m_\lambda, \lambda\in\Lambda.
\end{equation*}
\
{\it Necessary part.}

 Let $X=\{(\lambda,m_\lambda)\}_{\lambda\in\Lambda}$ be an interpolating divisor and assume that the  discs $\Big\{ D\left(\lambda,\sqrt{\frac{m_\lambda-C_X}{m_\lambda+\alpha+1}}\right)\Big\}_{\lambda,\ m_\lambda>C_X}$ are not separated for any $C_X>0$. Let $r_{\lambda}=\sqrt{\frac{m_{\lambda}-C_X}{m_{\lambda}+\alpha+1}}$ and $r_{\lambda}'=\sqrt{\frac{m_{\lambda}-(C_X-1)}{m_{\lambda}+\alpha +1}}$. 
 There  exists $\lambda_1,\lambda_2\in\Lambda$, $\lambda_1\neq\lambda_2$ 
and $ m_{\lambda_1},m_{\lambda_2}>C_X$  such that 
\begin{equation*}
  D\left(\lambda_1,r_{\lambda_1}\right)\cap  D\left(\lambda_2,r_{\lambda_2}\right)\neq \emptyset.
\end{equation*}
And also by the same argument 
\begin{equation*}
  D\left(\lambda_1,\sqrt{\frac{m_{\lambda_1}-(C_X-1)}{m_{\lambda_1}+\alpha+1}}\right)\cap  D\left(\lambda_2,\sqrt{\frac{m_{\lambda_2}-(C_X-1)}{m_{\lambda_2}+\alpha+1}}\right)\neq \emptyset.
\end{equation*}

Let $\zeta\in\partial D\left(\lambda_1,r_{\lambda_1}\right)$ and $\zeta'\in\partial D\left(\lambda_1,r_{\lambda_1}'\right)$,
 we have 
\begin{multline*}
\rho\left(\zeta,\zeta'\right) \geq \frac{\rho(\zeta' ,\lambda_1 )-\rho(\lambda_1,\zeta )} {1-\rho(\zeta',\lambda_1) \rho(\lambda_1,\zeta ) } 
=\rho(r_\lambda,r_{\lambda}') = \frac{  r_{\lambda_1}'-r_{\lambda_1} }{ 1-r_{\lambda_1}r_{\lambda_1}'} 
                  \geq   \frac{(r_{\lambda_1}')^2-(r_{\lambda_1})^2  }{ 1-(r_{\lambda_1}r_{\lambda_1}')^2}\\
                        =   \frac{ m_{\lambda_1}+\alpha+1 }{ m_{\lambda_1}(2\alpha+2C_X+1)+(\alpha+2)^2-C_X^2+C_X}
           \geq   \frac{1}{ 2\alpha+2C_X+1}=:\delta>0.                  
\end{multline*}
Hence, the  estimate of the hyperbolic distance between $\partial D\left(\lambda_1,r_{\lambda_1}\right)$ and $\partial D\left(\lambda_1,r_{\lambda_1}'\right)$, 
is bounded from below by $\delta$.
Thus,  if $w\in D\left(\lambda_1,r_{\lambda_1}\right)  \cap D\left(\lambda_2,r_{\lambda_2}\right)\subset  D\left(\lambda_1,r_{\lambda_1}'\right)\cap D\left(\lambda_2,r_{\lambda_2}'\right)$, and 
\[ 
\varepsilon=\frac{1}{2}\delta,
\]
then we have 
\begin{equation*}
 D(w,\varepsilon) \subset \left(\lambda_1,\sqrt{\frac{m_{\lambda_1}-C_X}{m_{\lambda_1}+\alpha+1}}\right)\cap  D\left(\lambda_2,\sqrt{\frac{m_{\lambda_2}-C_X}{m_{\lambda_2}+\alpha+1}}\right).
\end{equation*}

Since $X$ is an interpolating divisor, there exists $f\in\Bergman$ such that 
\begin{center}
\begin{enumerate}[label=(\alph*)]
	\item $f\in N_{\lambda_1,m_{\lambda_1}}^2$,
	\item $f-T_w1\in  N_{\lambda_2,m_{\lambda_2}}^2$,
	\item $\|f\|_{\alpha,2}\leq M_X$, where $M_X$ is a fixed (interpolating) constant
depending only on $X$.
\end{enumerate}
\end{center}
By Lemma \ref{lemmamenorradio} applied to both $f$ and $f-T_w1$ (for which the sum
of the squares of the corresponding Fourier coefficients $|\langle g,e_j\rangle|$ vanish and the
norm on the disks are in particular bounded by $M_X$) we have
	\begin{align*}
	&\int_{ D\left(\lambda_1,\sqrt{\frac{m_{\lambda_1}-C_X+1}{m_{\lambda_1}+\alpha+1}}\right)}|f(z)|^2dA_{\alpha}(z)\\
	&+\int_{ D\left(\lambda_2,\sqrt{\frac{m_{\lambda_2}-C_X+1}{m_{\lambda_2}+\alpha+1}}\right)}|(f-T_w1)(z)|^2 dA_{\alpha}(z)= o(1)\cdot M_X^2,\quad C_X\rightarrow\infty,
	\end{align*}
	and therefore
	\begin{align*}
	\int_{D\left(w,\epsilon\right)}|f(z)|^2dA_{\alpha}(z)&+\int_{D\left(w,\epsilon\right)}|(f-T_w1)(z)|^2dA_{\alpha}(z)\\
	& =o(1)\cdot M_X^2,\quad C_X\rightarrow\infty.
	\end{align*}
	On the other hand,
	\begin{align*}
	(\alpha+1)\int_{D\left(w,\epsilon\right)}|(T_w1)(z)|^2dA_{\alpha}(z)&=(\alpha+1)\int_{ D\left(0,\epsilon\right)}(1-|z|^2)^\alpha dm(z)\\
	&=\left(1-\left(1-\epsilon^2\right)^{\alpha+1}\right)>0,
	\end{align*}
	which gives a contradiction when $C_X>C_X(M_X)$. 

\end{proof}

{
\section{The $\mathcal{A}_\alpha^\infty$-case}\label{section3}

Let $\alpha >0$, we now consider 
\begin{equation*}
    \mathcal{A}^{\infty}_\alpha=\left\{f\in \Hol
(\mathbb{D}):\|f\|_{\alpha,\infty}:=\sup_{z\in\mathbb{D}} (1-|z|^{2})^{\frac{\alpha}{2}}|f(z)|<+\infty\right\}.
\end{equation*}  

We shall start recalling the reformulation of interpolation and
sampling met in the situation $p=2$ in terms of vanishing subspaces. 

For each $\lambda,m_{\lambda}$ we have already introduced the subspace
\begin{equation*}
N_{\lambda}^2:=N^{2,\alpha}_{\lambda,m}=\{f\in\Bergman: f^{(j)}(\lambda)=0, \quad\forall j<m\}.
\end{equation*}
Observe that $\sum_{j<m_\lambda}\left|\langle f,T_\lambda e_j\rangle\right|^2
=\| f\|^2_{\Bergman/N_\lambda^2}$.
Then it becomes clear that $X$ is a \textit{sampling divisor} for $\Bergman$ if, for all $f\in\Bergman$,
\begin{equation*}
\|f\|_{\alpha,2}^2\simeq\sum_{\lambda\in\Lambda}\sum_{j<m_\lambda}\left|\langle f,T_\lambda e_j\rangle\right|^2=\sum_{\lambda\in\Lambda}\| f\|^2_{\Bergman/N_\lambda^2}.
\end{equation*}
Similarly, $X$ is \textit{interpolating} for $\Bergman$, if for all sequence $(f_\lambda)_{\lambda\in\Lambda}\subset\Bergman$ such that
\begin{equation*}
\sum_{\lambda\in\Lambda}\|f_\lambda\|_{\Bergman/N^2_\lambda}^2<\infty,
\end{equation*}
there exists $f\in\Bergman$ so that
\begin{equation*}
f-f_\lambda\in N_\lambda^2.
\end{equation*}

In order to consider the corresponding $L^\infty$ sampling and interpolation problems, we associate to each $\lambda\in\U$ the subspace
\begin{equation*}
	N_\lambda^\infty= N_{\lambda,m_\lambda}^{\infty,\alpha}:=\{f\in\Bergmaninf:\partial^j f(\lambda)=0, \forall j<m_\lambda \}.
\end{equation*}
\begin{definition}
A divisor is called sampling for $\Bergmaninf$, if there exists $L>0$ such that
\begin{equation*}
	\|f\|_{\alpha,\infty}\leq L\sup_{\lambda\in\Lambda}\|f\|_{\Bergmaninf/N_\lambda^\infty}.
\end{equation*}
\end{definition}
In a similar way we define generalized interpolation.
\begin{definition}
The divisor $X$ is called interpolating for $\Bergmaninf$ if for every sequence $(f_\lambda)_{\lambda\in\Lambda}$ with
\begin{equation*}
	\sup_{\lambda\in\Lambda}\|f_\lambda\|_{\Bergmaninf/N_\lambda^\infty}<\infty,
\end{equation*}
there exists a function $f\in\Bergmaninf$ such that
\begin{equation*}
	f-f_\lambda\in N_\lambda^\infty,\quad\lambda\in\Lambda.
\end{equation*}
\end{definition}

\subsection{Local $L^\infty$-estimates}
As in the $L^2$ case we need a local control of the functions of the space $\Bergmaninf$ with small quotient norm. Here is the result corresponding to Lemma \ref{lemmamenorradio} for
$\mathcal{A}^{\infty}_{\alpha}$ ($\alpha>0$).
\begin{lemma}\label{lem:controlnorminf}
	(i) For every $\eta,\varepsilon\in (0,1)$, there exists $C>0$ 
such that if $f\in\Bergmaninf{}$ satisfies $\|f\|_{\alpha,\infty}\leq 1$, $m\geq C$, $\|f\|_{\Bergmaninf/N_{0,m}^{\infty,\alpha}}<\varepsilon$, then 
	\begin{equation*}
	|f(z)|\left(1-|z|^2\right)^{\frac{\alpha}{2}}\leq \eta+\varepsilon , \quad z\in  D\left(0,\sqrt{\frac{m-C}{m+\alpha}}\right).
	\end{equation*}

(ii) For every $C>0$ there exist $\eta,\varepsilon\in (0,1)$, 
such that if $f\in\Bergmaninf{}$ satisfies $\|f\|_{\alpha,\infty}\leq 1$, $m\geq C$, $\|f\|_{\Bergmaninf/N_{0,m}^{\infty,\alpha}}<\varepsilon$, then 
	\begin{equation*}
	|f(z)|\left(1-|z|^2\right)^{\frac{\alpha}{2}}\leq 1-\eta , \quad z\in  D\left(0,\sqrt{\frac{m-C}{m+\alpha}}\right).
	\end{equation*}
\end{lemma}

The result in (i) is of course of interest when $\eta+\varepsilon<1$, and in particular
when $f\in N_{0,m}^{\infty,\alpha}$ in which case we can pick $\varepsilon$ arbitrarily 
small.

Note that the critical radius $\sqrt{\frac{m}{m+\alpha}}$ is different from the one
appearing for $p=2$. We have already met this radius in Theorem \ref{thmcovering}.

\begin{proof}
\underline{Claim (i):}

	Since $\|f\|_{\Bergmaninf/N_{0,m}^{\infty,\alpha}}<\varepsilon$, there exist  a function $g\in N_{0,m}^{\infty,\alpha}$ such that
	\begin{enumerate}[label=(\alph*)]
		\item $\|f-g\|_{\Bergmaninf}\leq\varepsilon,$
		\item $g(z)=z^m h(z),$ where $h$ is a holomorphic function in the unit disk.
	\end{enumerate}
	Since $\|f\|_{\alpha,\infty}\leq 1$ we have the bound
	$$
	|g(z)|\leq |g(z)-f(z)|+|f(z)|\leq(1+\varepsilon)\frac{1}{(1-|z|^2)^{\frac{\alpha}{2}}}, \qquad z\in\mathbb{D}.
	$$
	and in terms of $h$ and the functions $\vartheta_{m_,\alpha}(t)=\log\dfrac{1}{t^m(1-t^2)^\frac{\alpha}{2}}$ 
	$$	
	|h(z)|\leq (1+\varepsilon)e^{\vartheta_{m_,\alpha}(|z|)}, \qquad z\in\mathbb{D}.
	$$
	Using the maximum principal 	we obtain
\[ 
	\max_{z\in D\left(0,\sqrt{\frac{m}{m+\alpha }}\right)}|h(z)|
	=(1+\varepsilon)e^{ \left[\vartheta_{m_,\alpha}\left(\sqrt{\frac{m}{m+\alpha }}\right)-\vartheta_{m_,\alpha}\left(\sqrt{\frac{m-C}{m+\alpha }}\right)+\vartheta_{m_,\alpha}\left(\sqrt{\frac{m-C}{m+\alpha }}\right)\right]}
\]
	Observe that
\begin{equation}\label{p32*} 
	    \vartheta_{m,\alpha}\left(\sqrt{\frac{m-C}{m+\alpha }}\right)-\vartheta_{m,\alpha}\left(\sqrt{\frac{m}{m+\alpha }}\right)
	      =      \frac{C}{2}+o(1)-\log  \left(\frac{\alpha+2C}{\alpha}\right)^{\frac{\alpha}{2}}.
\end{equation}
Since the term $o(1)$ goes to 0 when $m$ goes to infinity, and $m\ge C$, the
above expression can be made arbitrarily big. Let $\delta=\delta(C)$ be the 
corresponding constant (thus with $\lim_{C\to+\infty}\delta(C)=+\infty$),
we get 
	\begin{equation}\label{p32**}
	    \max_{z\in D\left(0,\sqrt{\frac{m}{m+\alpha }}\right)}|h(z)|\leq (1+\varepsilon)e^{-\delta}e^{ \vartheta_{m_,\alpha}\left(\sqrt{\frac{m-C}{m+\alpha }}\right)}.
	\end{equation}

	Since $\delta(C)\to+\infty$ when $C\to+\infty$ there exists a $C$ such that
$(1+\varepsilon)e^{-\delta}\leq \eta$. Then 
	\begin{equation*}
	|h(z)|\leq \eta e^{ \vartheta_{m_,\alpha}\left(\sqrt{\frac{m-C}{m+\alpha }}\right)}, \quad z\in \partial D\left(0,\sqrt{\frac{m}{m+\alpha }}\right).
	\end{equation*}
Now, by the maximum principal again, restricting the estimate to the smaller disk
$D\left(0,\sqrt{\frac{m-C}{m+\alpha }}\right)$, and using the fact that $\vartheta_{m,\alpha}$
is decreasing on $(0,\sqrt{\frac{m}{m+\alpha}})$ we get
\begin{equation}\label{eqn:boundofh}
    	|h(z)|\leq \eta e^{ \vartheta_{m_,\alpha}\left(\sqrt{\frac{m-C}{m+\alpha }}\right)}
 \le \eta e^{\vartheta_{m,\alpha}(|z|)}, \quad z\in  D\left(0,\sqrt{\frac{m-C}{m+\alpha }}\right).
\end{equation}
	Finally, for $z\in D\left(0,\sqrt{\frac{m-C}{m+\alpha }}\right)$
	\begin{align}\label{p32***}
	|f(z)|\left(1-|z|^2\right)^{\frac{\alpha}{2}}&\leq |f(z)-g(z)|\left(1-|z|^2\right)^{\frac{\alpha}{2}}+|g(z)|\left(1-|z|^2\right)^{\frac{\alpha}{2}}\nonumber\\
	&\leq\varepsilon +|z|^m\left(1-|z|^2\right)^{\frac{\alpha}{2}}|h(z)|\nonumber\\
	&=\varepsilon +e^{-\vartheta_{m_,\alpha}\left(|z|\right)}|h(z)|\nonumber\\
	&\leq\varepsilon+\eta.
	\end{align}

\underline{Claim (ii):}

The proof follows exactly the same lines and ideas.
First one should observe that given $C>0$, the difference appearing in \eqref{p32*} is
uniformly bounded from below by some $\delta>0$ (this is clear when $m$ is big, say
$m\ge m_0$, and for $1\le m<m_0$ we just take the smallest of finitely many 
stricly positive numbers). Then looking at \eqref{p32**}, we have to 
convince ourselves that there are $\varepsilon,\eta>0$ such that $(1+\varepsilon)e^{-\delta}
<1-\eta-\varepsilon$ which is easily seen to be true. 
Finally, the same estimates as in \eqref{p32***} lead to
\[
|f(z)|\left(1-|z|^2\right)^{\frac{\alpha}{2}}
 \le \varepsilon+(1-\eta-\varepsilon)=1-\eta.
\]
\end{proof}

\subsection{Sampling for $\mathcal{A}_\alpha^\infty$}
Now we are ready to establish our conditions for sampling conditions.
\begin{theorem}
Let $\alpha>0$.
\begin{enumerate}[label=(\alph*)]
\item If $X$ is a sampling divisor for $\Bergmaninf$, then there exists $0<C<\alpha$ such that
\begin{equation*}
\bigcup_{\lambda\in\Lambda}  D\left(\lambda,\sqrt{\frac{m_\lambda+C}{m_\lambda+\alpha }}\right)=\mathbb{D}.
\end{equation*}
\item Conversely, if there exists $C=C(S_X)>0$ such that for some compact $K$ of $\mathbb{D}$ we have
\begin{equation*}
\bigcup_{\lambda\in\Lambda, m_{\lambda>C}}  D\left(\lambda,\sqrt{\frac{m_\lambda-C}{m_\lambda+\alpha }}\right)=\mathbb{D}\setminus K,
\end{equation*}
then $X$ is a sampling divisor for $\Bergmaninf$.
\end{enumerate}
\end{theorem}

\begin{proof}\ \\
(a)\textit{Necessary Condition.}\\
{%
 Suppose that for every $ 0<C<\alpha$, we have 
\begin{equation*}
\bigcup_{\lambda\in\Lambda}  D\left(\lambda,\sqrt{\frac{m_\lambda+C}{m_\lambda+\alpha }}\right)\neq\mathbb{D}.
\end{equation*}
Thus, there exists an increasing sequence of positive numbers $(C_k)$ tending to $\alpha$
and a sequence $(z_k)$ with $z_{k}\in\mathbb{D}$ such that :

\begin{equation*}
z_{k}\in\mathbb{D}  \setminus \left[  \bigcup_{\lambda\in\Lambda}  D\left(\lambda,\sqrt{\frac{m_\lambda+C_{k}}{m_\lambda+\alpha }}  \right)\right]  \subset 
\mathbb{D}\setminus \left[ \bigcup_{\lambda\in\Lambda}  D\left(\lambda,\sqrt{\frac{m_\lambda}{m_\lambda+\alpha }}\right)   \right].
\end{equation*}

 As in the proof of the necessary condition of the sampling theorem in the Hilbertian case,  we will show that
\[d_k:=\dist\Big(z_k, \bigcup_{\lambda\in\Lambda}  D\big(\lambda, \sqrt{\frac{m_\lambda}{m_\lambda+\alpha}   }\big)\Big) \rightarrow 1.\]
 Put 
 $   r_{\lambda, C_{k}}=\sqrt{\frac{m_{\lambda}+C_{k}}{m_{\lambda}+\alpha }}$ and $r_{\lambda}=\sqrt{\frac{m_{\lambda}}{m_{\lambda}+\alpha }}$. Let
\[    \zeta \in\bigcup_{\lambda\in\Lambda}  D\left(\lambda,\sqrt{\frac{m_\lambda}{m_\lambda+\alpha}}\right). \]
 Then there exists $\lambda_0\in \Lambda$ such that $\zeta\in D(\lambda_0,r_{\lambda_0})$ and $\rho(z_k,\lambda_{0}  ) > r_{\lambda_{0}, C_{k}}$, $k\geq 1$. Since $\alpha>0$,
we get as in the hilbertian situation
\begin{eqnarray*}
\rho(z_k,\zeta) & \geq & \frac{  \rho(z_k,\lambda_{0}  )-\rho(\zeta,\lambda_{0}) } {1-\rho(z_k,\lambda_{0}) \rho(\zeta,\lambda_{0})}  
   \geq   \frac{ (  r_{\lambda_{0}, C_{k}})^{2} - (r_{\lambda_{0}})^{2}  } {1-( r_{\lambda_0})^2( r_{\lambda_{0}, C_{k}} )^{2} }     \\
   & =&   \frac{  C_k( m_{\lambda_0}+\alpha )  } {  m_{\lambda_0}(2\alpha-C_k)+\alpha^2 }.
\end{eqnarray*}
Observe that this last expression is decreasing in $m_{\lambda_0}$ so that passing
to the limit $m_{\lambda_0}\to+\infty$, we get
\[
\rho(z_k,\zeta)\ge \frac{C_k}{2\alpha-C_k}.
\]
Thus, $d_k\ge C_k/(2\alpha-C_k)\to 1$ when $k\to+\infty$.
}

Now pick $f_k(z)=T_{z_k}(1)$. We will show that we cannot sample uniformly $f_k$
meaning that $\sup_{\lambda\in\Lambda}\|f_k\|_{\mathcal{A}^{\infty}_{\alpha}/
N_{\lambda,m_{\lambda}}^{\alpha,\infty}} \to 0$ (while $\|f_k\|_{\mathcal{A}^{\infty}_{\alpha}}=1$).
In view of the construction it is enough to show that when $u\ge\sqrt{\frac{m+C_k}{m+\alpha}}$, then
$\|T_{u}1\|_{\mathcal{A}^{\infty}_{\alpha}/N_{0,m}^{\alpha,\infty}} \to 0$ uniformly in $m$
when $C_k\to \alpha$.
Recall that
\[
 T_u1(z)=\left(\frac{1-|u|^2}{(1-\overline{u}z)^2}\right)^{\alpha/2} =
 (1-|u|^2)^{\alpha/2}\frac{1}{(1-\overline{u}z)^{\alpha}}.
\] 
Using the standard Taylor series for power functions we get
\[
\frac{1}{(1-\overline{u}z)^{\alpha} }
 =\sum_{n\ge 0}\binom{-\alpha}{n}(-\overline{u})^nz^n
 =\sum_{n=0}^{m-1}\binom{-\alpha}{n}(-\overline{u})^nz^n+z^mh(z)
 =f_0(z)+z^mh(z).
\]
The following etimate is well known
\[
 \binom{-\alpha}{n}=\frac{(-1)^n}{\Gamma(\alpha)n^{1-\alpha}}(1+o(1)).
\]
Hence
\[
 |f_0(z)| 
 \le C \sum_{n= 0}^{m-1}\frac{1}{n^{1-\alpha}}u^n|z|^n
\]
(we remind that $u>0$). Here $C$ is some irrelevant universal constant.
Hence
\[
\|T_{u}1\|_{\mathcal{A}^{\infty}_{\alpha}/N_{0,m}^{\alpha,\infty}}
 \le (1-|u|^2)^{\alpha/2}\|f_0\|_{\mathcal{A}^{\infty}}
 \le C(1-|u|^2)^{\alpha/2}
  \sup_{|z|<1}(1-|z|^2)^{\alpha/2}\sum_{n= 0}^{m-1}\frac{1}{n^{1-\alpha}}u^n|z|^n
\]
The function $\varphi_n(x)=(1-x^2)^{\alpha/2}x^n$ admits a maximum in 
$x_n=\sqrt{n/(n+\alpha/2)}$ which, up to a multiplicative constant, 
behaves like $1/n^{\alpha/2}$. Hence
\[
\|T_{u}1\|_{\mathcal{A}^{\infty}_{\alpha}/N_{0,m}^{\alpha,\infty}}
 \le C(1-|u|^2)^{\alpha/2} \sum_{n= 0}^{m-1}n^{\alpha/2-1}
 \le C(1-|u|^2)^{\alpha/2} m^{\alpha/2}
\]
where in the above inequalities $C$ are different universal constants.
On the other hand
\[
 (1-|u|^2)^{\alpha/2}\le\left(1-\frac{m+C_k}{m+\alpha}\right)^{\alpha/2}
 =\left(\frac{\alpha-C_k}{m+\alpha}\right)^{\alpha/2},
\]
so that
\[
\|T_{u}1\|_{\mathcal{A}^{\infty}_{\alpha}/N_{0,m}^{\alpha,\infty}}
 \le C\left(m\frac{\alpha-C_k}{m+\alpha}\right)^{\alpha/2}
 \le C(C_k-\alpha)^{\alpha/2}
\]
uniformly in $m$. Since
$C_k\to \alpha$ the above expression goes to 0 (uniformly in $m$), and we reach the
desired conclusion.
\\

(b)\textit{Sufficient Condition.}\\	
Suppose that there exists a sequence $(f_n)_n$ such that $\|f_n\|_{\alpha,\infty}=1$, and
\begin{equation*}
    \sup_{\lambda\in\Lambda}\|f_n\|_{\Bergmaninf/N_\lambda^\infty}\rightarrow0,\qquad n\rightarrow\infty.
\end{equation*}
 Passing to a sub-sequence converging uniformly on compact subsets denoted again by $(f_n)_n$, we have two possibilities: either (A) the sequence $(f_n)_n$ converges to $f\not=0$ or (B) the sequence $(f_n)_n$ converges to 0.\\

  (A): In this case $X$ is a zero divisor for $\Bergmaninf$. Then, by the Uniqueness Theorem \ref{thmcovering}, $\mathbb{D}\setminus
 \Big[ \bigcup_{\lambda\in\Lambda} D\left(\lambda,\sqrt{\frac{m_\lambda}{m_\lambda+\alpha+\varepsilon }}  \right)\Big] $ cannot be compact for any $\varepsilon>0$. On the other hand, for every $C>0$ and for every $0<\varepsilon<C$ we have
 \begin{equation}
 \sqrt{\frac{m_\lambda-C}{m_\lambda+\alpha }} < \sqrt{\frac{m_\lambda}{m_\lambda+\alpha+\varepsilon }}.
 \end{equation}
 This yields,
 \begin{equation*}
\mathbb{D}  \setminus \left[  \bigcup_{\lambda\in\Lambda}  D\left(\lambda,\sqrt{\frac{m_\lambda}{m_\lambda+\alpha+\varepsilon }}  \right)\right]  \subset 
\mathbb{D}\setminus \left[ \bigcup_{\lambda\in\Lambda}  D\left(\lambda,\sqrt{\frac{m_\lambda-C}{m_\lambda+\alpha }}\right)   \right].
\end{equation*} 
Therefore, for no $C>0$,  $\mathbb{D}\setminus \left[ \bigcup_{\lambda\in\Lambda}  D\left(\lambda,\sqrt{\frac{m_\lambda-C}{m_\lambda+\alpha }}\right)   \right] $ can be compact, contradicting the hypothesis. \\  

   (B): In this case, by contradiction we will assume that for some compact set $K\subset \mathbb{D}$ we have 
\begin{equation*}
    \Omega= \bigcup_{\lambda\in\Lambda} D\Big(\lambda,\sqrt{\frac{m_\lambda-C}{m_\lambda+\alpha }}\Big)=\mathbb{D}\setminus K. 
\end{equation*}
Since by assumption $(f_n)_n$ converges to 0 on compact subsets, there exists $n_0\in\mathbb{N}$ such that
\begin{equation*}
    \left|f_n(z)\right|\left(1-|z|^2\right)^{\frac{\alpha}{2}}<\frac{1}{2},\qquad z\in K,n\geq n_0.
\end{equation*}

Next, for the given $C$, Lemma \ref{lem:controlnorminf}(ii) implies the existence of
$\eta, \varepsilon>0$ ensuring a control on $f$.
Since  $ \sup_{\lambda\in\Lambda}\|f_n\|_{\Bergmaninf/N_\lambda^\infty}\rightarrow0$, 
there exists $n_1$ such that for $n\ge n_1$, these quotient norms are stricly smaller than
$\varepsilon$ (uniformly in $\lambda$) as required by the  lemma.
Since moreover $ \| f_{n}\|_{\alpha,\infty} =1$, 
Lemma \ref{lem:controlnorminf} implies that 
\[\left|f_n(z)\right|\left(1-|z|^2\right)^{\frac{\alpha}{2}}<1-\eta, \quad z\in{\bigcup_{\lambda\in\Lambda, m_{\lambda>C}}} D\Big(\lambda,\sqrt{\frac{m_\lambda-C}{m_\lambda+\alpha }}\Big).\]  
Hence, $\|f_n\|_{\alpha,\infty}<1$ for $n>\max(n_0,n_1)$ and we get a contradiction.

\end{proof}

}

{

\subsection{Interpolation for $\mathcal{A}_\alpha^\infty$}

We need the following result by  Berndtsson \cite[Theorem 4]{B} (see \cite[Theorem G]{BO95} ) for the uniform estimates in the $\dbar$-surgery.
\begin{theorem}\label{UnifBerndtsson}
Let $\psi$ be a subharmonic function and
\begin{equation*}
    \varphi(z)=\min\left\{(1-|z|)\Delta\psi(z),\frac{1}{1-|z|}\right\}.
\end{equation*}
Let $f$ be a function in $\U$ such that
\begin{equation*}
    \sup\frac{|f(z)|}{\varphi(z)}e^{-\psi(z)/2}<\infty.
\end{equation*}
Let $u\in L^2(\mathbb{D}, e^{-\psi}dm)$ be the canonical solution to $\dbar u=f$. Then 
\begin{equation*}
    \sup|u(z)|e^{-\widetilde\psi(z)/2}\leq \sup\frac{|f(z)|}{\varphi(z)}e^{-\psi(z)/2},
\end{equation*}
where $\widetilde{\psi}(z)=\sup_{|z-\zeta|<1/2(1-|z|)} \psi(\zeta)$.
\end{theorem}
The corresponding result for interpolation in $\mathcal{A}^{\infty}_{\alpha}$ reads as 
follows.
\begin{theorem}
Let $\alpha>0$.
\begin{enumerate}[label=(\alph*)]
\item 
If $X$ is an interpolating divisor for $\Bergmaninf$, then there exists $C_X>0$ such that the hyperbolic disks $$\left\{ D\left(\lambda,\sqrt{\frac{m_\lambda-C_X}{m_\lambda+\alpha}}\right)\right\}_{\lambda\in\Lambda,m_\lambda>C_X}$$ are pairwise disjoint.
\item 
Conversely, if for some  $C_X$ such that   $\alpha(1-e^{-1})<C_X<\alpha$,  the hyperbolic disks    $$ \left\{ D\left(\lambda,\sqrt{\frac{m_\lambda+C_X}{m_\lambda+\alpha}}\right)\right\}_{\lambda\in\Lambda}$$ are pairwise disjoint, then $X$ is an  interpolating divisor for $\Bergmaninf$.
\end{enumerate}
\end{theorem}
\begin{proof}
{\it Necessary part}

 Let $X=\{(\lambda,m_\lambda)\}_{\lambda\in\Lambda}$ be an interpolating divisor and assume that the discs $$\left\{ D\left(\lambda,\sqrt{\frac{m_\lambda-C_X}{m_\lambda+\alpha}}\right)\right\}_{\lambda\in\Lambda, m_\lambda>C_X}$$ are not pairwise disjoint for any $C_X$. Arguing as in the proof of Theorem  \ref{ThmInter2}, we see that there exist $\lambda, \lambda'\in\Lambda$ and $w\in\U$ such that
\begin{equation*}
     D\left(w,\epsilon\right)\subset  D\left(\lambda,\sqrt{\frac{m_\lambda-C_X+1}{m_\lambda+\alpha}}\right)\cap  D\left(\lambda',\sqrt{\frac{m_{\lambda'}-C_X+1}{m_{\lambda'}+\alpha}}\right).
\end{equation*}
Since $X$ is an interpolating divisor, there exists a function $f\in\Bergmaninf$ such that
	\begin{center}
	\begin{enumerate}[label=(\alph*)]
		\item $f\in N_{\lambda,m_\lambda}^\infty$,
		\item $f-T_w 1\in N_{\lambda',m_{\lambda'}}^\infty$,
		\item $\|f\|_{\alpha,\infty}\leq M_X$.
	\end{enumerate}
	\end{center}
Let us denote $\|\cdot\|_{\infty,U}$ the norm with a supremum taken in the set $U\subset\U$. By Lemma \ref{lem:controlnorminf}(i) applied to $f$ and $f-T_w 1$ we have 
\begin{equation*}
	\|f\|_{\infty,  D\left(\lambda,\sqrt{\frac{m_\lambda-C_X+1}{m_\lambda+\alpha}}\right)}+\|f-T_w1\|_{\infty,  D\left(\lambda',\sqrt{\frac{m_{\lambda'}-C_X+1}{m_{\lambda'}+\alpha}}\right)}<2\eta, 
\end{equation*}
where we can pick $\eta<1/2$ when $C_X$ is sufficiently big (note that since $f$ and $f-T_w1$
are zero in the corresponding quotient spaces, we can consider $\varepsilon=0$).
Therefore
\begin{equation*}
\|f\|_{\infty, D(w,\varepsilon)}+\|f-T_w1\|_{\infty, D(w,\varepsilon)}<2\eta
\end{equation*}
However
\begin{equation*}
\sup_{z\in D(w,\varepsilon)}\|T_w1\|_{\infty, D(w,\varepsilon)}=1,
\end{equation*}
so $X$ cannot be interpolating.\\

{\it Sufficient part}

Here we use the same scheme as in the $L^2$-case: we construct a smooth interpolating function and we modify it to obtain a holomorphic one. However, now we need an $L^\infty$-estimate for the solution to the $\dbar$-equation which will be provided by Theorem \ref{UnifBerndtsson}.

Let $(\rho_{\lambda_j})_{j\geq1}$ be holomorphic data (polynomials) with $\sup_j \|\rho_{\lambda_j}\|_{\alpha,\infty} \leq 1$. Given any $N\geq 1$ we look for functions $f_N\in\Hol(\U)$ and $M$ independent of $N$ such that
\begin{align*}
f_N-\rho_j\in N_{\lambda_j}^\infty,& \quad j=1,\ldots,N;\\
\|f_N\|_{\alpha,\infty}\leq M,& \quad \forall N\in\N.
\end{align*}
Then, by Montel's theorem, the limit $f=\lim_{N} f_N$ gives the desired result. 
For this set, $m_j=m_{\lambda_j}$, 
$$
D_j=D\left(\lambda_j,r_j\right), \quad r_j=\sqrt{\frac{m_j}{m_j+\alpha}} \quad \text{ and } \quad 
D_j^{'}=D\left(\lambda_j,r_j'\right), \quad r_j'= \sqrt{\frac{m_j+C_X}{m_j+\alpha}}.
$$

Define  the  smooth interpolating function,
\begin{equation*}
F_N(z)=\sum_{j=1}^N\rho_j(z)\eta(|\varphi_{\lambda_j}(z)|-r_j'),
\end{equation*}
where  $\eta=\eta_\lambda$ is a smooth cut-off function on $\R$, with 
\begin{enumerate}[label=(\alph*)]
	\item $\text{supp }\eta\subset(-\infty,0],$\label{etaprop1}
	\item $\eta\equiv 1$ on $(-\infty,r_j-r_j'],$ \label{etaprop2}
	\item $|\eta'|\lesssim \frac{1}{r_j'-r_j}\simeq \frac{m_j+\alpha}{C_X}$. \label{etaprop3}
\end{enumerate}

By the separation hypothesis we have
 \[ \text{supp } F_{N} \subset \bigcup_{j=1}^{N}  D_j' \subset \bigcup_{j\geq 1} D_j'.\]
 Furthermore, $F_N$ has the characteristic growth of $\Bergmaninf$, due to the property \ref{etaprop2} and the separation hypothesis again, Namely 
 \begin{align}\label{F_N growth}
    \sup_{z\in\mathbb{D}} \left(1-|z|^2\right)^{\frac{\alpha}{2}} |F_N(z)|&=\max_{1\leq j\leq N}\sup_{z\in\mathbb{D}} \left(1-|z|^2\right)^{\frac{\alpha}{2}} |\rho_j(z)|\\
  &  \leq \sup_{j\geq 1}  \|\rho_j\|_{\Bergmaninf}\leq 1
 \end{align}
uniformly in $N$. On the other hand, 
\begin{equation*}
\text{supp }\dbar F_{N} \subset \bigcup_{j=1}^{N} \left( D_j' \setminus  D_j \right).
\end{equation*}
And 
$$
\dbar F_N(z)=\sum_{j\leq N} \rho_j(z)\eta'(|\varphi_{\lambda_j}(z)|-r_j') \dbar {|\varphi_{\lambda_j}(z)|}\chi_{D_j' \setminus  D_j }(z).
$$
Hence, for $z\in D_j'\setminus  D_j$, as in \eqref{p25*},
\begin{align*}
    |\overline{\partial}F_{N}(z)| &=|\rho_j(z)|\left| \eta' \left(|\varphi_{\lambda_j}(z)|-r_j' \right) \right|  \left| \overline{\partial}{ |\varphi_{\lambda_j}(z)}\right|\\
    &\leq \frac{1}{2} |\rho_j(z)|  \|\eta'\|_{\infty} \frac{ 1-|\varphi_{\lambda_j}(z)|^2  } {1-|z|^2} .
\end{align*}
Therefore, by  \ref{etaprop2} and \ref{etaprop3}, we get for $z\in D_j'\setminus  D_j$ 
\begin{align*}
    |\overline{\partial}F_{N}(z)| \left(1-|z|^2\right)^{\frac{\alpha}{2}+1}  & \lesssim \|\rho_j\|_{\alpha,\infty}  \|\eta'\|_{\infty} ( 1-|\varphi_{\lambda_j}(z)|^2 ) \\
       &\lesssim \|\rho_j\|_{\alpha,\infty}   \frac{1-r_{j}^{2}}{r_j'-r_j}\\
      & \lesssim \frac{2(\alpha+1)}{C_X} \|\rho_j\|_{\alpha,\infty},
\end{align*}
where we have used a similar estimate as in \eqref{p27*}.

This leads to
\begin{equation}\label{new_weighted}
   \sup_{z\in\mathbb{D}} \frac{ |\overline{\partial}F_{N}(z)|  }{\frac{1}{1-|z|^2}}  e^{-\frac{\alpha}{2}\log(\frac{1}{1-|z|^{2}})}\lesssim\sup_{j\geq 1}  \|\rho_j\|_{\Bergmaninf}<\infty,
\end{equation}
where underlying constants are independant on $N$.

Again the holomorphic interpolating function in $\mathcal{A}^{\infty}_{\alpha}$ will be obtained {\it via} the solution to a  $\overline{\partial}$-problem:
$f_{N}=F_{N}-u_{N}$, where $\overline{\partial} u_{N}=\overline{\partial} F_{N}$ with the conditions
\begin{equation*}
\sup_{z\in\mathbb{D}} |u_N(z)| \left(1-|z|^2\right)^{\frac{\alpha}{2}}<\infty,
\end{equation*}
and
\begin{equation*}
\partial^{k} u_N(\lambda_{j})=0,\ \forall k<m_{\lambda_{j} }.
\end{equation*}
The last condition ensures that $\partial^{k} f_{N}(\lambda_{j})=\partial^{k} F_{N}(\lambda_{j})$,  for  $k<m_{\lambda_{j}}$,  and then 
\begin{equation}\label{Derivativeufrho}
 \partial^{k} \big( F_{N}-u_{N}-\rho_{j}\big)(\lambda_j)=0, \quad k<m_{\lambda_{j}}, 1\leq j\leq N.
\end{equation}

{
We will use a similar weight function $w$ as in \eqref{p23*} 
where now}
\[
 r_\lambda=\sqrt{\frac{m_\lambda}{m_\lambda+\alpha}} , 
\quad \text{and}\quad  r_\lambda'= \sqrt{\frac{m_\lambda+C_X}{m_\lambda+\alpha}}.
\]
More precisely, set
\begin{equation*}
 w:=w_{\Lambda,\alpha,N}(z)=\sum_{j=1}^{N} m_{\lambda_j} \left[ E(.) - E\star\xi_{\lambda_j}(.)\right]  (\varphi_{\lambda_j}(z)) \chi_{D_{\lambda_j} '}(z),\quad z\in\U, 
\end{equation*} 
where   $E(z)=\log|z|^{2}$, and for  $\lambda \in\Lambda$
$$
\xi_{\lambda}(\zeta)= \left\{
    \begin{array}{lll}
             0  &   \mbox{if }  0\leq|\zeta|< r_{\lambda}, \\
         \dfrac{\log \frac{1}{|\zeta|^{2}} }{ K(m_\lambda,C_X,\alpha)}  & \mbox{if }    r_{\lambda}<|\zeta|<  r_{\lambda}^{'}, \\
            0               &\mbox{if }  |\zeta|>  r_{\lambda}^{'}.
    \end{array}
\right.
$$
Again $K:=K(m_{\lambda},C_X,\alpha)=\int_{r_\lambda<|\zeta|<r'_\lambda}\log (1/|\zeta|^2)d\V(\zeta)$. As in \eqref{p24*} we see that $K\ge \log(\alpha/(\alpha-C_X))>1$ when
$C_X>\alpha(1-e^{-1})$. Let 
$$
\psi(z):=\psi_{\Lambda,\alpha}(z)=\alpha \log\frac{1}{1-|z|^2} +w(z), \quad z\in\U.
$$
By the same arguments as in the proof of lemma \ref{TheWeightedW}, we have 
\begin{itemize} 
 \item[(a)]  $w\leq 0$,  
 \item[(b)] $-w\leq A(\alpha)$ in $D'_j\setminus D_j$,
 \item[(c)] $\widetilde{\Delta} w \geq -4(\alpha-\varepsilon)$ for some 
$\varepsilon$ depending on $C_X$.
\end{itemize}
Clearly from (a) and the definition of $\psi$ we have
$ \psi(z)\leq \alpha \log\frac{1}{1-|z|^2}$ for $z\in \U$, and with (c) we get
that under the condition $ \alpha (1-e^{-1})<C_X<\alpha$, for every $z\in\U$, 
\[ 
(1-|z|^2)\Delta\psi(z)=\frac{ \widetilde{\Delta} \psi (z) }{1-|z|^{2}} \gtrsim \frac{\varepsilon(C_X)}{1-|z|^{2}}.  
\] 
Thus,  
$$
\varphi(z):=\min\{  (1-|z|)\Delta\psi(z), \frac{1}{1-|z|} \} \asymp \frac{1}{1-|z|^{2}},\qquad z\in\U. 
$$

Now applying Theorem \ref{UnifBerndtsson}, we see that the there exists $u_N\in L^2(\U, e^{-\psi}dm)$, a canonical solution of the $\dbar$-equation $\dbar u_{N}=\dbar F_{N}$ satisfying
$$
\sup_{z\in\U} |u_{N}(z)|e^{-\frac{1}{2}\widetilde{\psi(z)}} \leq \sup_{z\in\U} \frac{|\dbar F_{N}(z)|}{\varphi(z)}   e^{  -\frac{1}{2}\psi(z)  },
$$ 
where $\widetilde{\psi(z)}:= \sup_{{|z-\zeta|<1/2(1-|z|)}}  \psi(z)$. By \eqref{new_weighted} and (b)
\begin{align*}
\sup_{z\in\U} \frac{|\dbar F_{N}(z)|}{\varphi(z)}   e^{  -\frac{1}{2}\psi(z)  } & \lesssim \sup_{z\in\U}  \frac{ |\overline{\partial} F_{N}(z)|  }{\frac{1}{1-|z|^{2}} }  e^{-\frac{\alpha}{2}\log(\frac{1}{1-|z|^{2}})-w(z)}\\
&\lesssim \sup_{j=1,\ldots,N}\sup_{z\in D'_j\setminus D_j }    \frac{ |\overline{\partial}F_{N}(z)|  }{ \frac{1}{1-|z|^2} }  e^{-\alpha \log\frac{1}{1-|z|^2}}e^{A(\alpha)}\\
&\lesssim\sup_{j\geq 1}  \|\rho_j\|_{\Bergmaninf}\lesssim 1.
\end{align*}
Thus, uniformly  in $N$
$$
\sup_{z\in\U} |u_{N}(z)|e^{-\frac{1}{2}\widetilde{\psi(z)}} <\infty.
$$ 
On the other hand,  by (a) 
\begin{align*}
\widetilde{\psi(z)}:= \sup_{{|z-\zeta|<1/2(1-|z|)}}  \psi(\zeta) &=\sup_{{|z-\zeta|<1/2(1-|z|)}}  \left(  \alpha\log\frac{1}{1-|\zeta|^2} +  w(\zeta)\right) \\
 &\leq   \alpha \log\frac{1}{1-|z|^2}  + \log\frac{3}{2}.
\end{align*} 

We obtain finally
\begin{equation}
   \sup_{z\in\U} |u_{N}(z)| e^{- \frac{\alpha}{2} \log\frac{1}{1-|z|^2}  } \leq  \frac{2}{3} \sup_{z\in\U} |u_{N}(z)|e^{-\frac{1}{2}\widetilde{\psi}(z) }\lesssim 1.
\end{equation}

Hence, by (\ref{F_N growth}) and (\ref{Derivativeufrho}), $f_{N}=F_N-u_N\in \mathcal{A}^{\infty}_{\alpha}$. 
This completes the proof.
\end{proof}
}

\end{document}